\documentclass[a4paper,11pt]{amsart}
\title[Finitely generated congruences]{Finitely generated congruences on tropical rational function semifields}

\author{Song JuAe}
\address{Department of Mathematics, Graduate School of Science, Kyoto University, Kitashirakawa Oiwake-cho, Sakyo-ku, Kyoto 606-8502, Japan.}
\email{song.juae.8m@kyoto-u.ac.jp}


\usepackage{amsmath,amssymb,amscd}
\usepackage{amsthm}
\usepackage{mathrsfs}

\newtheorem{dfn}{Definition}[section]
\newtheorem{thm}[dfn]{Theorem}
\newtheorem{prop}[dfn]{Proposition}
\newtheorem{cor}[dfn]{Corollary}
\newtheorem{lemma}[dfn]{Lemma}

\def\Gamma{\varGamma}

\begin{document}

\maketitle

\begin{abstract}
We prove that the congruence on the tropical rational function semifield in $n$-variables associated with a subset $V$ of $\boldsymbol{R}^n$ is finitely generated if and only if the closure of $V$ is a finite union of $\boldsymbol{R}$-rational polyhedral sets.
With this fact, we characterize rational function semifields of tropical curves.
\end{abstract}

\section{Introduction}
	\label{section1}

The name of tropical geometry was given in \cite{Sturmfels}.
It is an algebraic geometry over the tropical semifield $\boldsymbol{T} := ( \boldsymbol{R} \cup \{ -\infty \}, \oplus := \operatorname{max}, \odot := +)$ and has many sources; e.g., Bergman fans (\cite{Bergman}), amoebas (\cite{GKZ}), tropical algebra (introduced by Imre Simon, cf.~\cite{Simon}) and patchworkings (cf.~\cite{Viro}).
In the last two decades, it has been revealed that tropical geometry has many applications in other areas.
For example, Mikhalkin's correspondence Principle (\cite{Mikhalkin}) states that the Gromov--Witten numbers can be found tropically.
Another proof of Brill--Noether theorem was given in \cite{CDPR}.
Tropical polynomials are used to describe the earliest finishing times of project networks in \cite{Kobayashi=Odagiri}.

By such backgrounds, it is expected that tropical geometry itself has an interest and rich structure.
To construct algebraic foundation for tropical geometry, Bertram and Easton (\cite{Bertram=Easton}) and Jo\'{o} and Mincheva (\cite{Joo=Mincheva1}) proved a tropical Nullstellensatz for congruences on tropical polynomial semirings $\boldsymbol{T}[X_1, \ldots, X_n]$.
As in the classical algebraic geometry, corresponding to a congruence $E$ on $\boldsymbol{T}[X_1, \ldots, X_n]$, the geometric object, called a congruence variety, $\boldsymbol{V}(E)_0 \subset \boldsymbol{T}^n$ is defined, and a subset $V$ of $\boldsymbol{T}^n$ defines the congruence $\boldsymbol{E}(V)_0$ on $\boldsymbol{T}[X_1, \ldots, X_n]$ (see Subsection \ref{subsection2.4} for more details).
With these concepts, the tropical Nullstellensatz gives a characterization of $\boldsymbol{E}(\boldsymbol{V}(E)_0)_0$ when $E$ is finitely generated as a congruence, i.e., states that $\boldsymbol{E}(\boldsymbol{V}(E)_0)_0$ coincides with the congruence $\widehat{E}$ defined from $E$ with specific conditions.
There is a problem that the tropical Nullstellensatz does not ensure that $\widehat{E}$ is also finitely generated as a congruence.

The current paper is devoted to solving this problem when $E$ is a finitely generated congruence on the tropical rational function semifield $\overline{\boldsymbol{T}(X_1, \ldots, X_n)}$.
More precisely, we prove the following theorem:

\begin{thm}
	\label{thm:main}
Let $V$ be a subset of $\boldsymbol{R}^n$.
Then $\boldsymbol{E}(V)$ is finitely generated if and only if the closure (in the usual sense) $\overline{V}$ of $V$ in $\boldsymbol{R}^n$ is a finite union of $\boldsymbol{R}$-rational polyhedral sets.
\end{thm}

Here $\boldsymbol{E}(V)$ denotes the congruence on $\overline{\boldsymbol{T}(X_1, \ldots, X_n)}$ associated with $V$.
In the proof, we obtain a specific generating set of $\boldsymbol{E}(V)$ consisting of only one element explicitly.
This theorem has the following several corollaries:

\begin{cor}
	\label{cor1}
For a congruence $E$ on $\overline{\boldsymbol{T}(X_1, \ldots, X_n)}$, if $E$ is finitely generated as a congruence, then so is $\boldsymbol{E}(\boldsymbol{V}(E))$.
\end{cor}

Here $\boldsymbol{V}(E)$ is the congruence variety associated with $E$ and is in $\boldsymbol{R}^n$.

\begin{cor}[Corollary \ref{cor7}]
	\label{cor2}
Let $V$ (resp. $W$) be a subset of $\boldsymbol{R}^n$ (resp. $\boldsymbol{R}^m$).
When the quotient $\boldsymbol{T}$-algebra $\overline{\boldsymbol{T}(X_1, \ldots, X_n)} / \boldsymbol{E}(V)$ is isomorphic to $\overline{\boldsymbol{T}(Y_1, \ldots, Y_m)} / \boldsymbol{E}(W)$ as a $\boldsymbol{T}$-algebra, $\boldsymbol{E}(V)$ is finitely generated as a congruence if and only if so is $\boldsymbol{E}(W)$.
\end{cor}

\begin{cor}[{cf.~\cite[Proposition 3.9(i)]{Joo=Mincheva2}}]
	\label{cor3}
Let $V$ be a finite union of $\boldsymbol{R}$-rational polyhedral sets in $\boldsymbol{R}^n$.
Then there exist $f_1, f_2 \in \overline{\boldsymbol{T}[X_1, \ldots, X_n]}$ such that $\boldsymbol{E}(V)_1 = \{ (g, h) \in \overline{\boldsymbol{T}[X_1, \ldots, X_n]} \times \overline{\boldsymbol{T}[X_1, \ldots, X_n]} \,|\, \exists k \in \boldsymbol{Z}_{>0}: (g \odot f_2^{\odot k}, h \odot f_2^{\odot k}) \in \langle (f_1, f_2) \rangle \}$ and $f_1(x) \ge f_2(x)$ for any $x \in \boldsymbol{T}^n$.
In particular, if $f_2$ is taken as a constant tropical monomial of a real number, then $\boldsymbol{E}(V)_1 = \langle (f_1, f_2) \rangle$ holds.
\end{cor}

Here $\boldsymbol{E}(V)_1$ is the congruence on the tropical polynomial function semiring $\overline{\boldsymbol{T}[X_1, \ldots, X_n]}$ associated with $V$ and $\langle (f_1, f_2) \rangle$ stands for the congruence on $\overline{\boldsymbol{T}[X_1, \ldots, X_n]}$ generated by $(f_1, f_2)$.

Theorem \ref{thm:main} also gives a characterization of rational function semifields $\operatorname{Rat}(\Gamma)$ of tropical curves $\Gamma$ (Corollary \ref{cor9}).
Note here that (parts of) \cite{JuAe2}, \cite{JuAe3}, \cite{JuAe4} and \cite{JuAe5} were devoted to obtain this characterization.
With this characterization, we know that if the rational map $\phi$ to $\boldsymbol{R}^n$ induced by the complete linear system associated with an effective divisor on $\Gamma$ is injective, then $\overline{\boldsymbol{T}(X_1, \ldots, X_n)} / \boldsymbol{E}(\operatorname{Im}(\phi))$ is isomorphic to $\operatorname{Rat}(\Gamma)$ as a $\boldsymbol{T}$-algebra (and hence as a semifield over $\boldsymbol{T}$) (Corollary \ref{cor11}), where $\operatorname{Im}(\phi)$ stands for the image of $\phi$.
This means that $\phi$ has all imformation of $\operatorname{Rat}(\Gamma)$ as a $\boldsymbol{T}$-algebra and all geometric imformation $\Gamma$ as a tropical curve by \cite[Corollary 1.2]{JuAe5}, intuitively, $\phi$ is a ``good'' embedding.

Theorem \ref{thm:main} is a tropical analogue of the Noether's property.
However there exists no conclusive definition of (abstract) tropical varieties, Theorem \ref{thm:main} ensures that arbitrary ``tropical" geometric object $V$ in $\boldsymbol{R}^n$ in the broadest sense is defined by a finite system of equalities of tropical rational functions and defines the finitely generated semifield $\overline{\boldsymbol{T}(X_1, \ldots, X_n)} / \boldsymbol{E}(V)$ over $\boldsymbol{T}$ (cf. the Structure Theorem for Tropical Geometry \cite[Theorem 3.3.5]{Maclagan=Sturmfels}).
We expect the current paper to become a new step toward the construction of algebraic foundation for tropical geometry.

Note here that a part of Theorem \ref{thm:main} is indicated in the introduction of \cite{Grigoriev}.

The rest of this paper is organized as follows.
In Section \ref{section2}, we give the definitions of polyhedral cones and polyhedral sets, semirings, algebras and semifields, congruences, tropical rational function semifields and congruence varieties, tropical curves, rational functions, morphisms between tropical curves, and linear systems.
In Section \ref{section3}, which is our main section, we prove the above assertions.

\section*{Acknowledgements}
The author thanks her supervisor Masanori Kobayashi, Kazuhiko Yamaki, Yasuhito Nakajima, Keita Goto, Masao Oi, Yuki Koto, Taro Sogabe and Takaaki Ito for their helpful comments and discussions.

\section{Preliminaries}
	\label{section2}

In this section, we recall several definitions which we need later.
We refer to \cite{Golan} (resp. \cite{Maclagan=Sturmfels}) for an introduction to the theory of semirings (resp. tropical geometry and polyhedral geometry) and employ definitions in \cite{Bertram=Easton}, \cite{JuAe4} and \cite{Joo=Mincheva1} (resp. \cite{JuAe1}, \cite{JuAe2} and \cite{Maclagan=Sturmfels}) with slight modifications partially.
Also, we employ Chan's definition of morphisms between tropical curves in \cite{Chan}.

\subsection{Polyhedral cones and polyhedral sets}
	\label{subsection2.1}

For any subset $S$ of $\boldsymbol{R}^n$, the \textit{polar cone} $S^{\ast}$ of $S$ is defined by the set consisting of all points $x$ in $\boldsymbol{R}^n$ such that $x \cdot y \le 0$ holds for any $y \in S$, where $\cdot$ denotes the standard inner product.
It is well-known that $S^{\ast}$ is a closed convex cone, i.e., $S^{\ast}$ is closed and for any $x, y \in S^{\ast}$ and $r \in \boldsymbol{R}_{\ge 0}$, both the segment $\overline{xy}$ and the point $rx$ are in $S^{\ast}$, and $(K^{\ast})^{\ast} = K$ holds for any closed convex cone $K$.
The following theorem is well-known and one of the keys to our main theorem:

\begin{thm}
	\label{thm:decomposition1}
Let $K$ be a nonempty closed convex cone.
For any point $x$ of $\boldsymbol{R}^n$, there exist unique points $y$ of $K$ and $z$ of $K^{\ast}$ satisfying $x = y + z$ and $y \cdot z = 0$.
\end{thm}

Note that in Theorem \ref{thm:decomposition1}, $\operatorname{dist}( x, y ) = \operatorname{dist} ( x, K )$ holds, where $\operatorname{dist}(x, y)$ (resp. $\operatorname{dist} ( x, K )$) denotes the distance between $x$ and $y$ (resp. $K$).
For any subset $S$ of $\boldsymbol{R}^n$, let $\operatorname{Cone}(S)$ stand for the convex cone hull of $S$, i.e., the set of all points $x$ of $\boldsymbol{R}^n$ such that there exist a finite number of nonnegative numbers $t_1, \ldots, t_k \ge 0$ and $k$ elements $s_1, \ldots, s_k$ of $S$ satisfying $x = t_1 s_1 + \cdots + t_k s_k$.
A \textit{polyhedral cone} is the convex cone hull of a finite number of points of $\boldsymbol{R}^n$, which is closed.
If $K$ is a polyhedral cone, then so is $K^{\ast}$.
Since $(K^{\ast})^{\ast} = K$, a polyhedral cone is the solution set of a system of (a finite number of) homogeneous linear inequalities.
For $a \in \boldsymbol{R}^n$ and $b \in \boldsymbol{R}$, if $K$ is contained in $\{ x \in \boldsymbol{R}^n \,|\, a \cdot x + b \ge 0 \}$, then $K \cap \{ x \in \boldsymbol{R}^n \,|\, a \cdot x + b = 0 \}$ is called a \textit{face} of $K$.
A \textit{proper face} of $K$ is a face other than $\varnothing, K$.
For two subsets $S_1, S_2$ of $\boldsymbol{R}^n$, let $S_1 + S_2$ denote the \textit{Minkowski sum} of $S_1$ and $S_2$, i.e., the set $\{ s_1 + s_2 \,|\, s_1 \in S_1, s_2 \in S_2 \}$.
For a subset $S$ of $\boldsymbol{R}^n$, let $\operatorname{Int}(S)$ and $\partial (S)$ stand for the relative interior of $S$ and the boundary of $S$, respectively.
Theorem \ref{thm:decomposition1} gives a decomposition of $\boldsymbol{R}^n$ in the case that $K$ is a nonempty polyhedral cone:

\begin{cor}
	\label{cor:decomposition2}
For a finite number of points $x_1, \ldots, x_m$ of $\boldsymbol{R}^n \setminus \{ (0, \ldots, 0) \}$, let $K$ be the polyhedral cone $\{ y \in \boldsymbol{R}^n \,|\, \forall i, x_i \cdot y \le 0 \}$.
If $K$ is nonempty, then
\begin{align*}
\boldsymbol{R}^n &= K + K^{\ast}\\
&= K + \operatorname{Cone} \{ x_1, \ldots, x_m \}\\
&= \operatorname{Int}(K) \sqcup \bigsqcup_{H : \text{a proper face of }K} ( \operatorname{Int}(H) + \operatorname{Cone}\{ x_i \,|\, x_i \perp H \} )\\
&= \operatorname{Int}(K) \sqcup \bigsqcup_{y \in \partial (K)} (\{ y \} + \operatorname{Cone} \{ x_i \,|\, x_i \cdot y = 0 \}).
\end{align*}
\end{cor}

We extend Corollary \ref{cor:decomposition2} to the case of polyhedral sets.
Note that a \textit{polyhedral set} is the solution set of a system of (a finite number of) linear inequalities.
A polyhedral set is \textit{$\boldsymbol{R}$-rational} if all coefficients of its system of (a finite number of) linear inequalities except for constant terms are rational.

\begin{cor}
	\label{cor:decomposition3}
For a finite number of points $x_1, \ldots, x_m$ of $\boldsymbol{R}^n \setminus \{ (0, \ldots, 0) \}$ and real numbers $a_1, \ldots, a_m$, let $K$ be the polyhedral set $\{ y \in \boldsymbol{R}^n \,|\, \forall i, x_i \cdot y + a_i \le 0 \}$.
If $K$ is nonempty, then
\begin{align*}
\boldsymbol{R}^n = \operatorname{Int}(K) \sqcup \bigsqcup_{y \in \partial (K)} (\{ y \} + \operatorname{Cone} \{ x_i \,|\, x_i \cdot y + a_i = 0 \}).
\end{align*}
\end{cor}

\begin{proof}
It is enough to show that for any $z \in \boldsymbol{R}^n$ and the unique point $w$ of $K$ such that $\operatorname{dist}(z, w) = \operatorname{dist}(z, K)$ holds, $z$ is in $\{ w \} + \operatorname{Cone}\{ x_i \,|\, x_i \cdot w + a_i = 0 \}$.
We use the induction on $m$.
It is true for $m = 1$.
We assume that it is true for $m = k$.
We consider the case of $m = k + 1$.
For any $z \in \boldsymbol{R}^n$, let $w$ be the unique point of $K$ such that $\operatorname{dist}(z, w) = \operatorname{dist}(z, K)$ holds.
Assume that the cardinality of $\{ i \,|\, x_i \cdot w + a_i = 0 \}$ is less than $k + 1$, say $k + 1 \not\in \{ i \,|\, x_i \cdot w + a_i = 0 \}$.
Then $w$ is a point of the polyhedral set $L := \{ y \in \boldsymbol{R}^n \,|\, i = 1, \ldots, k, x_i \cdot y + a_i \le 0 \}$ and $\operatorname{dist}(z, w) = \operatorname{dist}(z, K) \ge \operatorname{dist}(z, L)$ and $x_{k + 1} \cdot w + a_{k + 1} < 0$.
There exists a unique point $w^{\prime}$ of $L$ such that $\operatorname{dist}(z, w^{\prime}) = \operatorname{dist}(z, L)$.
If $w \not= w^{\prime}$, then $\operatorname{dist}(z, w) > \operatorname{dist}(z, w^{\prime})$.
In this case, if $x_{k + 1} \cdot w^{\prime} + a_{k + 1} \le 0$, then $w^{\prime} \in K$, and hence $\operatorname{dist}(z, w^{\prime}) \ge \operatorname{dist}(z, K) = \operatorname{dist}(z, w)$, which is a contradiction.
Hence we have $x_{k + 1} \cdot w^{\prime} + a_{k + 1} > 0$.
Since both $w$ and $w^{\prime}$ are in $L$ and $L$ is convex, the segment $\overline{ww^{\prime}}$ is contained in $L$.
Thus there exists a unique point $w^{\prime \prime}$ of $\overline{ww^{\prime}}$ such that $x_{k + 1} \cdot w^{\prime \prime} + a_{k + 1} = 0$.
This $w^{\prime \prime}$ is a point of $K$.
As $\operatorname{dist}(z, w^{\prime}) < \operatorname{dist}(z, w)$, we have $\operatorname{dist}(z, w^{\prime \prime}) < \operatorname{dist}(z, w) = \operatorname{dist}(z, K)$, which is a contradiction.
Therefore $w^{\prime}$ must coincide with $w$.
By the assumption of the induction, we have $z \in \{ w \} + \operatorname{Cone} \{ x_i \,|\, x_i \cdot w + a_i = 0, i \not= k + 1 \} \subset \{ w \} + \operatorname{Cone} \{ x_i \,|\, x_i \cdot w + a_i = 0 \}$.
Assume that the cardinality of $\{ i \,|\, x_i \cdot w + a_i = 0 \}$ is $k + 1$.
Since $m = k + 1$, we have $K = \{ w \} + \{ y \in \boldsymbol{R}^n \,|\, \forall i, x_i \cdot y \le 0 \}$.
By Corollary \ref{cor:decomposition2}, $z - w \in \operatorname{Cone}\{ x_1, \ldots, x_{k + 1} \}$, and thus $z \in \{ w \} + \operatorname{Cone}\{ x_1, \ldots, x_{k + 1} \}$.
\end{proof}

A convex cone in $\boldsymbol{R}^n$ is \textit{strongly convex} if it contains no lines through the origin $O$ of $\boldsymbol{R}^n$.
For a polyhedral cone $K$ in $\boldsymbol{R}^n$, it is well-known that $K$ is strongly convex if and only if $K^{\ast}$ is $n$-dimensional and that if $K$ is not strongly convex, then there exist $\{ i_1, \ldots, i_s \} \subset \{ 1, \ldots, m \}$ and $c_{i_1}, \ldots, c_{i_s} > 0$ such that $\sum_{j = 1}^s c_{i_j}x_{i_j} = O$ holds.

\begin{lemma}
	\label{lem:distance}
For a finite number of points $x_1, \ldots, x_m$ of $\boldsymbol{R}^n \setminus \{ (0, \ldots, 0) \}$, there exists $k > 0$ such that $\sqrt{y \cdot y} \le k \operatorname{max} \{ x_i \cdot y \,|\, i = 1, \ldots, m \}$ holds for any point $y$ of $\operatorname{Cone}\{ x_1, \ldots, x_m \}$.
\end{lemma}

\begin{proof}
Let $K$ be the polyhedral cone $\operatorname{Cone}\{ x_1, \ldots, x_m \}$ and $d : K \to \boldsymbol{R}_{\ge 0}; y \mapsto \sqrt{y \cdot y}$ and $d^{\prime} : K \to \boldsymbol{R}_{\ge 0}; y \mapsto \operatorname{max}\{ x_i \cdot y \,|\, i = 1, \ldots, m \}$.
Note that $d^{\prime}(y)$ must be nonnegative for any $y \in K$.
In fact, if $y \in K$ satisfies $x_i \cdot y < 0$ for $i = 1, \ldots, m$, then $y \in K^{\ast}$, and hence $y \in K \cap K^{\ast}$ and $y$ is not $O$.
However, by Theorem \ref{thm:decomposition1}, $K \cap K^{\ast}$ is consiting only of $O$, which is a contradiction.

Assume that $K$ is strongly convex.
Let $y$ be a point of $\operatorname{Int}(K)$.
By the assumption above, $L := K  \cap \partial (\{ y \} + K^{\ast})$ is a nonempty compact subset of $\boldsymbol{R}^n$.
Since the functions $d$ and $d^{\prime}$ are continuous, there exist the maximum value of $d$ on $L$ and the minimum value $D^{\prime}$ of $d^{\prime}$ on $L$.
As $y \in \operatorname{Int}(K)$, $L$ does not contain $O$, and hence $D^{\prime}$ is positive.
Hence there exists $k > 0$ such that $d(z) \le k d^{\prime}(z)$ holds for any $z \in L$.
For $y^{\prime} := ty$ with $t > 0$ and $z^{\prime} \in K \cap \partial ( \{ y^{\prime} \} + K^{\ast})$, since $\frac{1}{t} z^{\prime} \in L$, we have $d(z^{\prime}) = t d \left( \frac{1}{t} z^{\prime} \right) \le t k d^{\prime} \left( \frac{1}{t} z^{\prime} \right) = kd^{\prime}(z^{\prime})$.
As $d(O) = d^{\prime}(O) = 0$, we have the conclusion in this case.

Assume that $K$ is not strongly convex.
Then there exist $\{ i_1, \ldots, i_s \} \subset \{ 1, \ldots, m \}$ and $c_{i_1}, \ldots, c_{i_s} > 0$ such that $\sum_{j = 1}^s c_{i_j}x_{i_j} = O$ holds.
It is clearly true that $K \supset \bigcup_{j = 1}^s \operatorname{Cone}\{ x_l \,|\, l \not= i_j \}$.
Conversely, for any $y \in K$, there exist $t_1, \ldots, t_m \ge 0$ such that $y = \sum_{i = 1}^m t_i x_i$.
Without loss of generality, we can assume that $i_j = j$, $c_{i_j} = 1$ for $j = 1, \ldots, s$ and $\operatorname{min} \{ t_{i_j} \,|\, j = 1, \ldots, s \} = t_{i_1}$.
Then we have $y = \sum_{i = 1}^m t_i x_i = \sum_{i = 1}^s (t_i - t_1) x_i + \sum_{i = s + 1}^m t_i x_i \in \operatorname{Cone}\{ x_2, \ldots, x_m \}$.
Thus we have $K = \bigcup_{j = 1}^s \operatorname{Cone}\{ x_l \,|\, l \not= i_j \}$.
By repeating this process at most finitely many times, $K$ is subdivided into finitely many strongly convex polyhedral cones $H_j$ generated by subsets $X_{H_j}$ of $\{ x_1, \ldots, x_m \}$ for $j = 1, \ldots, p$.
By the previous argument, there exists $k_{H_j} > 0$ such that $d(y) \le k_{H_j} \operatorname{max} \{ x_i \cdot y \,|\, x_i \in X_{H_j} \}$ holds for any point $y$ of $H_j$.
Since $\operatorname{max}\{ x_i \cdot y \,|\, x_i \in X_{H_j} \} \le \operatorname{max}\{ x_i \cdot y \,|\, i = 1, \ldots, m \}$, for $k := \operatorname{max}\{ k_{H_j} \,|\, j = 1, \ldots, p \}$, the inequality $d \le kd^{\prime}$ holds on $K$, which completes the proof.
\end{proof}

A vector $\boldsymbol{v} \in \boldsymbol{R}^n$ is \textit{primitive} if all its components are integers and their greatest common divisor is one.
When $\boldsymbol{v}$ is primitive, for $\lambda \ge 0$, the \textit{lattice length} of $\lambda \boldsymbol{v}$ is defined as $\lambda$ (cf.~\cite[Subsection 2.1]{JuAe1}).

\subsection{Semirings, algebras and semifields}
	\label{subsection2.2}

In this paper, a \textit{semiring} $S$ is a commutative semiring with the absorbing neutral element $0_S$ for addition $+$ and the identity $1_S$ for multiplication $\cdot$.
If every nonzero element of a semiring $S$ is multiplicatively invertible and $0_S \not= 1_S$, then $S$ is called a \textit{semifield}.

A map $\varphi : S_1 \to S_2$ between semirings is a \textit{semiring homomorphism} if for any $x, y \in S_1$,
\begin{align*}
\varphi(x + y) = \varphi(x) + \varphi(y), \	\varphi(x \cdot y) = \varphi(x) \cdot \varphi(y), \	\varphi(0) = 0, \	\text{and}\	\varphi(1) = 1.
\end{align*}
Given a semiring homomorphism $\psi : S_1 \to S_2$, if both $S_1$ and $S_2$ are semifield and $\psi$ is injective, then $S_2$ is a \textit{semifield over $S_1$}.

Given a semiring homomorphism $\varphi : S_1 \to S_2$, we call the pair $(S_2, \varphi)$ (for short, $S_2$) a \textit{$S_1$-algebra}.
For a semiring $S_1$, a map $\psi : (S_2, \varphi) \to (S_2^{\prime}, \varphi^{\prime})$ between $S_1$-algebras is a \textit{$S_1$-algebra homomorphism} if $\psi$ is a semiring homomorphism and $\varphi^{\prime} = \psi \circ \varphi$.
When there is no confusion, we write $\psi : S_2 \to S_2^{\prime}$ simply.
A bijective $S_1$-algebra homomorphism $S_2 \to S_2^{\prime}$ is a \textit{$S_1$-algebra isomorphism}.
Then $S_2$ and $S_2^{\prime}$ are said to be \textit{isomorphic}.

The set $\boldsymbol{T} := \boldsymbol{R} \cup \{ -\infty \}$ with two tropical operations:
\begin{align*}
a \oplus b := \operatorname{max}\{ a, b \} \quad	\text{and} \quad a \odot b := a + b,
\end{align*}
where $a, b \in \boldsymbol{T}$ and $a + b$ stands for the usual sum of $a$ and $b$, becomes a semifield.
Here, for any $a \in \boldsymbol{T}$, we handle $-\infty$ as follows:
\begin{align*}
a \oplus (-\infty) = (-\infty) \oplus a = a \quad \text{and} \quad a \odot (-\infty) = (-\infty) \odot a = -\infty.
\end{align*}
This triple $(\boldsymbol{T}, \odot, \oplus)$ is called the \textit{tropical semifield}.
The subset $\boldsymbol{B} := \{ 0, -\infty \}$ of $\boldsymbol{T}$ becomes a semifield with tropical operations of $\boldsymbol{T}$ and is called the \textit{boolian semifield}.
A $\boldsymbol{B}$-algebra $S$ is said to be \textit{cancellative} if whenever $x \cdot y = x \cdot z$ for some $x, y, z \in S$, then either $x = 0_S$ or $y = z$.
If $S$ is cancellative, then we can define the semifield $Q(S)$ of fractions as within the case of integral domains.
In this case, the map $S \to Q(S); x \mapsto x/1_S$ becomes an injective $\boldsymbol{B}$-algebra homomorphism.
A $\boldsymbol{B}$-algebra $S$ has a natural partial order, i.e., $x \ge y$ if and only if $x + y = x$ for $x, y \in S$.

The \textit{tropical polynomials} are defined in the usual way and the set of all tropical polynomials in $n$-variables is denoted by $\boldsymbol{T}[X_1, \ldots, X_n]$.
It becomes a semiring with two tropical operations and is called the \textit{tropical polynomial semiring}.
Similarly, we define the \textit{tropical Laurent polynomials} and the \textit{tropical Laurent polynomial semiring} $\boldsymbol{T}\left[X_1^{\pm 1}, \ldots, X_n^{\pm 1}\right]$.
By \cite[Example 2.1]{JuAe4}, we know that tropical (Laurent) polynomial semirings are not cancellative.

For a commutative monoid $(M, +)$ is a \textit{$\boldsymbol{T}$-module} if $(M, +)$ is equipped with a map $\mu : \boldsymbol{T} \times M \to M$ satisfying

$(1)$ $\mu(t, m_1 + m_2) = \mu(t, m_1) + \mu(t, m_2)$,

$(2)$ $\mu(t_1 \oplus t_2, m) = \mu(t_1, m) + \mu(t_2, m)$,

$(3)$ $\mu(t_1 \odot t_2, m) = \mu(t_1, \mu(t_2, m))$,

$(4)$ $\mu(0, m) = m$, and 

$(5)$ $\mu(-\infty, m) = 0_M$,

\noindent
where $t, t_1, t_2 \in \boldsymbol{T}$, $m, m_1, m_2 \in M$ and $0_M$ stands for the unity of $M$ (see \cite[Subsection 2.1]{Giansiracusa=Giansiracusa1} for more details).

\subsection{Congruences}
	\label{subsection2.3}

A \textit{congruence} $E$ on a semiring $S$ is a subset of $S^2 = S \times S$ satisfying

$(1)$ for any $x \in S$, $(x, x) \in E$,

$(2)$ if $(x, y) \in E$, then $(y, x) \in E$,

$(3)$ if $(x, y) \in E$ and $(y, z) \in E$, then $(x, z) \in E$,

$(4)$ if $(x, y) \in E$ and $(z, w) \in E$, then $(x + z, y + w) \in E$, and

$(5)$ if $(x, y) \in E$ and $(z, w) \in E$, then $(x \cdot z, y \cdot w) \in E$.

The diagonal set of $S^2$ is denoted by $\Delta$ and called the \textit{trivial} congruence on $S$.
It is the unique smallest congruence on $S$.
The set $S^2$ becomes a semiring with the operations of $S$ and is a congruence on $S$.
This is called the \textit{improper} congruence on $S$.
Congruences other than the improper congruence are said to be \textit{proper}.
A congruence $E$ on $S$ is \textit{finitely generated} as a congruence if there exist $(x_1, y_1), \ldots, (x_m, y_m) \in E$ such that the smallest congruence $\langle (x_1, y_1), \ldots, (x_m, y_m) \rangle_S = \langle (x_1, y_1), \ldots, (x_m, y_m) \rangle$ on $S$ containing all $(x_1, y_1), \ldots, (x_m, y_m)$ is $E$.
Quotients by congruences can be considered in the usual sense and the quotient semiring of $S$ by the congruence $E$ is denoted by $S / E$.
Then the natural surjection $\pi_E : S \twoheadrightarrow S / E$ is a semiring homomorphism.

For a semiring homomorphism $\psi : S_1 \to S_2$, the \textit{kernel congruence} $\operatorname{Ker}(\psi)$ of $\psi$ is the congruence $\{ (x, y) \in S_1^2 \,|\, \psi(x) = \psi(y) \}$.
For semirings and congruences on them, the fundamental homomorphism theorem holds (\cite[Proposition 2.4.4]{Giansiracusa=Giansiracusa1}).
Then, for the above $\pi_E$, we have $\operatorname{Ker}(\pi_E) = E$.

\subsection{Tropical rational function semifields and congruence varieties}
	\label{subsection2.4}

For $V \subset \boldsymbol{T}^n$, the set $\boldsymbol{E}(V)_0 := \{ (f, g) \in \boldsymbol{T}[X_1, \ldots, X_n]^2 \,|\, \forall x \in V, f(x) = g(x) \}$ is a congruence on $\boldsymbol{T}[X_1, \ldots, X_n]$.
The semiring $\overline{\boldsymbol{T}[X_1, \ldots, X_n]} := \boldsymbol{T}[X_1, \ldots, X_n] / \boldsymbol{E}(\boldsymbol{T}^n)_0$ is cancellative by \cite[Theorem 1]{Bertram=Easton} and \cite[Proposition 5.5 and Theorem 4.14(v)]{Joo=Mincheva1}.
We call it the \textit{tropical polynomial function semiring}.
We call its semifield of fractions the \textit{tropical rational function semifield} and write it $\overline{\boldsymbol{T}(X_1, \ldots, X_n)}$.
In what follows, by abuse of notation, we write the image of $X_i$ in $\overline{\boldsymbol{T}(X_1, \ldots, X_n)}$ as $X_i$.
For a subset $E$ of $\overline{\boldsymbol{T}(X_1, \ldots, X_n)}^2$, we define $\boldsymbol{V}(E) := \{ x \in \boldsymbol{R}^n \,|\, \forall (f, g) \in E, f(x) = g(x)\}$ and call it the \textit{congruence variety} associated with $E$.
By defining that subsets of $\boldsymbol{R}^n$ are closed if they are of the form of congruence varieties, a topology on $\boldsymbol{R}^n$ is determined and coincides with the Euclidean topology on $\boldsymbol{R}^n$ by \cite[Proposition 3.9]{JuAe4}.
We also prepare the definition $\boldsymbol{V}(F)_0 := \{ x \in \boldsymbol{T}^n \,|\, \forall (f, g) \in F, f(x) = g(x) \}$ for a subset $F$ of $\boldsymbol{T}[X_1, \ldots, X_n]^2$ and the congruence $\boldsymbol{E}(V)_1 := \{ (f, g) \in \overline{\boldsymbol{T}[X_1, \ldots, X_n]}^2 \,|\, \forall x \in V, f(x) = g(x) \}$ on $\overline{\boldsymbol{T}[X_1, \ldots, X_n]}$ associated with $V \subset \boldsymbol{T}^n$.

\subsection{Tropical curves}
	\label{subsection2.5}

In this paper, a \textit{graph} is an unweighted, undirected, finite, connected nonempty multigraph that may have loops.
For a graph $G$, the set of vertices is denoted by $V(G)$ and the set of edges by $E(G)$.
A vertex $v$ of $G$ is a \textit{leaf end} if $v$ is incident to only one edge and this edge is not a loop.
A \textit{leaf edge} is an edge of $G$ incident to a leaf end.

A \textit{tropical curve} is the topological space associated with the pair $(G, l)$ of a graph $G$ and a function $l: E(G) \to {\boldsymbol{R}}_{>0} \cup \{\infty\}$, where $l$ can take the value $\infty$ only on leaf edges, by identifying each edge $e$ of $G$ with the closed interval $[0, l(e)]$.
The interval $[0, \infty]$ is the one-point compactification of the interval $[0, \infty)$.
We regard $[0, \infty]$ not just as a topological space but as an extended metric space.
The distance between $\infty$ and any other point is infinite.
When $l(e)=\infty$, the leaf end of $e$ must be identified with $\infty$.
If $E(G) = \{ e \}$ and $l(e)=\infty$, then we can identify either leaf ends of $e$ with $\infty$.
When a tropical curve $\Gamma$ is obtained from $(G, l)$, the pair $(G, l)$ is called a \textit{model} for $\Gamma$.
There are many possible models for $\Gamma$.
We frequently identify a vertex (resp. an edge) of $G$ with the corresponding point (resp. the corresponding closed subset) of $\Gamma$.
A model $(G, l)$ is \textit{loopless} if $G$ is loopless.
For a point $x$ of a tropical curve $\Gamma$, if $x$ is identified with $\infty$, then $x$ is called a \textit{point at infinity}, else, $x$ is called a \textit{finite point}.
Let $\Gamma_{\infty}$ denote the set of all points at infinity of $\Gamma$.
If $x$ is a finite point, then the \textit{valency} $\operatorname{val}(x)$ is the number of connected components of $U \setminus \{ x \}$ with any sufficiently small connected neighborhood $U$ of $x$; if $x$ is a point at infinity, then $\operatorname{val}(x) := 1$.
The \textit{genus} $g(\Gamma)$ of $\Gamma$ is the first Betti number of $\Gamma$, which coincides with $\# E(G) - \# V(G) + 1$ for any model $(G, l)$ for $\Gamma$.
If $g(\Gamma) = 0$, then $\Gamma$ is a \textit{tree}.
The word ``an edge of $\Gamma$" means an edge of $G$ for some model $(G, l)$ for $\Gamma$.
A \textit{subgraph} of a tropical curve is a closed subset of the tropical curve with a finite number of connected components.

\subsection{Rational functions}
	\label{subsection2.6}

Let $\Gamma$ be a tropical curve.
A continuous map $f : \Gamma \to \boldsymbol{R} \cup \{ \pm \infty \}$ is a \textit{rational function} on $\Gamma$ if $f$ is a constant function of $-\infty$ or a piecewise affine function with integer slopes, with a finite number of pieces and that can take the value $\pm \infty$ at only points at infinity.
For a finite point $x$ of $\Gamma$ and a rational function $f$ on $\Gamma$ other than $-\infty$, for one outgoing direction at $x$, the \textit{outgoing slope} of $f$ at $x$ is $\frac{f(y) - f(x)}{\operatorname{dist}(x, y)}$ with a finite point $y$ in a sufficiently small neighborhood of $x$ in the dirextion.
If $x$ is a point at infinity, then we regard the outgoing slope of $f$ at $x$ as the slope of $f$ from $y$ to $x$ times minus one, where $y$ is a finite point on the leaf edge incident to $x$ such that $f$ has a constant slope on the interval $(y, x)$.
In both cases, the definition of outgoing slope of $f$ at $x$ in the direction is independent of the choice of $y$.
Note when $\Gamma$ is a singleton, i.e., consists of only one point, for any rational function $f$ other than $-\infty$, we define $\operatorname{ord}_x(f) := 0$.
Let $\operatorname{ord}_x(f)$ denote the sum of outgoing slopes of $f$ at $x$.
Let $\operatorname{Rat}(\Gamma)$ denote the set of all rational functions on $\Gamma$.
For rational functions $f, g \in \operatorname{Rat}(\Gamma)$ and a point $x \in \Gamma \setminus \Gamma_{\infty}$, we define
\begin{align*}
(f \oplus g) (x) := \operatorname{max}\{f(x), g(x)\} \quad \text{and} \quad (f \odot g) (x) := f(x) + g(x).
\end{align*}
We extend $f \oplus g$ and $f \odot g$ to points at infinity to be continuous on the whole of $\Gamma$.
Then both are rational functions on $\Gamma$.
Note that for any $f \in \operatorname{Rat}(\Gamma)$, we have
\begin{align*}
f \oplus (-\infty) = (-\infty) \oplus f = f
\end{align*}
and
\begin{align*}
f \odot (-\infty) = (-\infty) \odot f = -\infty.
\end{align*}
Then $\operatorname{Rat}(\Gamma)$ becomes a semifield with these two operations.
Also, $\operatorname{Rat}(\Gamma)$ becomes a $\boldsymbol{T}$-algebra and a semifield over $\boldsymbol{T}$ with the natural inclusion $\boldsymbol{T} \hookrightarrow \operatorname{Rat}(\Gamma)$.
Note that for $f, g \in \operatorname{Rat}(\Gamma)$, $f = g$ means that $f(x) = g(x)$ for any $x \in \Gamma$.

\subsection{Morphisms between tropical curves}
		\label{subsection2.7}

Let $\varphi : \Gamma \to \Gamma^{\prime}$ be a continuous map between tropical curves.
Then $\varphi$ is a \textit{morphism} if there exist loopless models $(G, l)$ and $(G^{\prime}, l^{\prime})$ for $\Gamma$ and $\Gamma^{\prime}$, respectively, such that $\varphi$ can be regarded as a map $V(G) \cup E(G) \to V(G^{\prime}) \cup E(G^{\prime})$ satisfying $\varphi(V(G)) \subset V(G^{\prime})$ and for $e \in \varphi(E(G))$, there exists a nonnegative integer $\operatorname{deg}_e(\varphi)$ such that for any points $x, y$ of $e$, $\operatorname{dist}_{\varphi(e)}(\varphi (x), \varphi (y)) = \operatorname{deg}_e(\varphi) \cdot \operatorname{dist}_e(x, y)$, where $\operatorname{dist}_{\varphi(e)}(\varphi(x), \varphi(y))$ (resp.~$\operatorname{dist}_e(x, y)$) denotes the distance between $\varphi(x)$ and $\varphi(y)$ in $\varphi(e)$ (resp.~$x$ and $y$ in $e$).
A map between tropical curves $\varphi : \Gamma \to \Gamma^{\prime}$ is an \textit{isomorphism} if $\varphi$ is a bijective morphism and the inverse map $\varphi^{-1}$ of $\varphi$ is also  a morphism.
By this definition, $\varphi$ is an isomorphism if and only if $\varphi$ is continuous on the whole of $\Gamma$ and is a local isometry on $\Gamma \setminus \Gamma_{\infty}$.
Then $\Gamma$ is \textit{isomorphic} to $\Gamma^{\prime}$.
For a morphism $\varphi : \Gamma \to \Gamma^{\prime}$, the \textit{pull-back} $\varphi^{\ast} : \operatorname{Rat}(\Gamma^{\prime}) \to \operatorname{Rat}(\Gamma); f^{\prime} \mapsto f^{\prime} \circ \varphi$ is a $\boldsymbol{T}$-algebra homomorphism (cf.~\cite[Proposition 3.1]{JuAe3}).

\subsection{Linear systems}
	\label{subsection2.8}

Let $\Gamma$ be a tropical curve.
An element of the free abelian group $\operatorname{Div}(\Gamma)$ generated by all points of $\Gamma$ is called a \textit{divisor} on $\Gamma$.
For a divisor $D$ on $\Gamma$, its \textit{degree} $\operatorname{deg}(D)$ is defined by the sum of the coefficients over all points of $\Gamma$.
We write the coefficient at $x$ as $D(x)$.
A divisor $D$ on $\Gamma$ is said to be \textit{effective} if $D(x) \ge 0$ for any $x$ in $\Gamma$.
For any $f \in \operatorname{Rat}(\Gamma) \setminus \{ -\infty \}$, the sum $\operatorname{ord}_x(f)$ is zero for all $x$ but a finite number of points of $\Gamma$, and thus
\[
\operatorname{div}(f) := \sum_{x \in \Gamma}\operatorname{ord}_x(f) \cdot x
\]
is a divisor on $\Gamma$, which is called the \textit{principal divisor} defined by $f$.
Remark that the constant function of $-\infty$ on $\Gamma$ does not define a principal divisor.
Two divisors $D$ and $E$ on $\Gamma$ are said to be \textit{linearly equivalent} if $D - E$ is a principal divisor.
For a divisor $D$ on $\Gamma$, the \textit{complete linear system} $|D|$ is defined by the set of effective divisors on $\Gamma$ which are linearly equivalent to $D$.
For a divisor $D$ on $\Gamma$, let $R(D)$ be the set of rational functions $f \in \operatorname{Rat}(\Gamma) \setminus \{ -\infty \}$ such that $D + \operatorname{div}(f)$ is effective together with $-\infty$.
\cite[Corollary 9]{Haase=Musiker=Yu} states that $R(D)$ is finitely generated as a $\boldsymbol{T}$-module and we can choose up extremals from every finite generating set to be a minimal generating set.
Here an \textit{extremal} of $R(D)$ is an element $f$ of $R(D)$ such that $f = g_1 \oplus g_2$ implies $f = g_1$ or $f = g_2$ for $g_1, g_2 \in R(D)$.
A characterization of extremals is given by \cite[Lemma 5]{Haase=Musiker=Yu}.

\section{Main results}
	\label{section3}

In this section, we prove all assertions in Section \ref{section1}.
To do so, we first deal with finitely generated congruences on $\overline{\boldsymbol{T}(X_1, \ldots, X_n)}$.
For a point $x \in \boldsymbol{R}^n$, let $(x)_i$ denote the $i$-th component of $x$.

\begin{lemma}
	\label{lem1}
For $(f, g) \in \overline{\boldsymbol{T}(X_1, \ldots, X_n)}^2$, the congruence variety $\boldsymbol{V}((f, g))$ is a finite union of $\boldsymbol{R}$-rational polyhedral sets.
\end{lemma}

\begin{proof}
If $\boldsymbol{V}((f, g))$ is empty or $\boldsymbol{R}^n$, then it is clear.

We assume that $\boldsymbol{V}((f, g))$ is not either empty or $\boldsymbol{R}^n$.
By assumption, $f, g \not= -\infty$ and $f \not= g$.
Since $f, g \in \overline{\boldsymbol{T}(X_1, \ldots, X_n)} \setminus \{ -\infty \}$, there exist $f_1, f_2, g_1, g_2 \in \overline{\boldsymbol{T}[X_1, \ldots, X_n]} \setminus \{ -\infty \}$ such that $f = f_1 \odot f_2^{\odot (-1)}$ and $g = g_1 \odot g_2^{\odot (-1)}$.
It is easy to check that for any $x \in \boldsymbol{R}^n$, $x \in \boldsymbol{V}((f, g))$ if and only if $x \in \boldsymbol{V}((f_1 \odot g_2, g_1 \odot f_2))$.
Hence we can assume that $f, g \in \overline{\boldsymbol{T}[X_1, \ldots, X_n]} \setminus \{ -\infty \}$.
Let $f^{\prime}$ and $g^{\prime}$ be representatives of $f$ and $g$ in $\boldsymbol{T}[X_1, \ldots, X_n] \setminus \{ -\infty \}$ resprectively.
By definition, we have $\boldsymbol{V}((f, g)) = \boldsymbol{V}((f^{\prime}, g^{\prime}))_0 \cap \boldsymbol{R}^n$.

Let $f^{\prime} = \bigoplus_{\boldsymbol{i} = (i_1, \ldots, i_n) \in A} a_{\boldsymbol{i}} \odot X_1^{\odot i_1} \odot \cdots \odot X_n^{\odot i_n}$ and $g^{\prime} = \bigoplus_{\boldsymbol{j} = (j_1, \ldots, j_n) \in B} b_{\boldsymbol{j}} \odot X_1^{\odot j_1} \odot \cdots \odot X_n^{\odot j_n}$ for some $a_{\boldsymbol{i}}, b_{\boldsymbol{j}} \in \boldsymbol{R}$ and some finite subsets $A, B \subset \left( \boldsymbol{Z}_{\ge 0} \right)^n$, where $\boldsymbol{i} \in A$ and $\boldsymbol{j} \in B$.
For any $x \in \boldsymbol{R}^n$, we have
\begin{alignat*}{2}
&&& x \in \boldsymbol{V}((f^{\prime}, g^{\prime}))_0\\
&\iff&& f^{\prime}(x) = g^{\prime}(x)\\
&\iff&& \text{there exist } \boldsymbol{i} = (i_1, \ldots, i_n) \in A, \boldsymbol{j} = (j_1, \ldots, j_n) \in B \text{ such that}\\
&&& \text{for any } \boldsymbol{k} = (k_1, \ldots, k_n) \in A, \boldsymbol{l} = (l_1, \ldots, l_n) \in B,\\
&&& \hspace{15pt} a_{\boldsymbol{i}} + i_1(x)_1 + \cdots + i_n(x)_n = b_{\boldsymbol{j}} + j_1(x)_1 + \cdots + j_n(x)_n\\
&&& \ge a_{\boldsymbol{k}} + k_1(x)_1 + \cdots + k_n(x)_n, b_{\boldsymbol{l}} + l_1(x)_1 + \cdots + l_n(x)_n\\
&\iff&& \text{there exist } \boldsymbol{i} = (i_1, \ldots, i_n) \in A, \boldsymbol{j} = (j_1, \ldots, j_n) \in B \text{ such that}\\
&&& \text{for any } \boldsymbol{k} = (k_1, \ldots, k_n) \in A, \boldsymbol{l} = (l_1, \ldots, l_n) \in B,\\
&&&(a_{\boldsymbol{i}} - b_{\boldsymbol{j}}) + (i_1 - j_1)(x)_1 + \cdots + (i_n - j_n)(x)_n = 0,\\ &&&(a_{\boldsymbol{i}} - a_{\boldsymbol{k}}) + (i_1 - k_1)(x)_1 + \cdots + (i_n - k_n)(x)_n \ge 0,\\
&&&(b_{\boldsymbol{j}} - b_{\boldsymbol{l}}) + (j_1 - l_1)(x)_1 + \cdots + (j_n - l_n)(x)_n \ge 0.
\end{alignat*}
Therefore we have
\begin{alignat*}{2}
&\boldsymbol{V}((f^{\prime}, g^{\prime}))_0 && \cap \boldsymbol{R}^n\\
=& \bigcup_{\substack{\boldsymbol{i} = (i_1, \ldots, i_n) \in A\\ \boldsymbol{j} = (j_1, \ldots, j_n) \in B}} \Bigg[ && \{ x \in \boldsymbol{R}^n \,|\, (a_{\boldsymbol{i}} - b_{\boldsymbol{j}}) + (i_1 - j_1)(x)_1 + \cdots + (i_n - j_n)(x)_n \ge 0 \}\\
&&& \cap \{ x \in \boldsymbol{R}^n \,|\, (a_{\boldsymbol{i}} - b_{\boldsymbol{j}}) + (i_1 - j_1)(x)_1 + \cdots + (i_n - j_n)(x)_n \le 0 \}\\
&&& \cap \bigcap_{\boldsymbol{k} = (k_1, \ldots, k_n) \in A} \{ x \in \boldsymbol{R}^n \,|\, (a_{\boldsymbol{i}} - a_{\boldsymbol{k}})+ (i_1 - k_1)(x)_1 + \cdots + (i_n - k_n)(x)_n \ge 0 \}\\
&&& \cap \bigcap_{\boldsymbol{l} = (l_1, \ldots, l_n) \in B} \{ x \in \boldsymbol{R}^n \,|\, (b_{\boldsymbol{j}} - b_{\boldsymbol{l}})+ (j_1 - l_1)(x)_1 + \cdots + (j_n - l_n)(x)_n \ge 0 \}\Bigg].
\end{alignat*}
This completes the proof.
\end{proof}

\begin{lemma}
	\label{lem2}
Let $E$ be a congruence on $\overline{\boldsymbol{T}(X_1, \ldots, X_n)}$.
For $f, g \in \overline{\boldsymbol{T}(X_1, \ldots, X_n)} \setminus \{ -\infty \}$, $(f, g) \in E$ if and only if $\left( f \odot g^{\odot (-1)} \oplus 0, 0 \right), \left(f^{\odot (-1)} \odot g \oplus 0, 0 \right) \in E$. 
\end{lemma}

\begin{proof}
First we note that since $f, g \not= -\infty$, the inverses $f^{\odot (-1)}, g^{\odot (-1)}$ exist.

We assume that $(f, g) \in E$ holds.
As $E$ is a congruence, we have
\begin{align*}
&\left( f \odot g^{\odot (-1)} \oplus 0, 0 \right) = \left( f \odot g^{\odot (-1)} \oplus 0, g \odot g^{\odot (-1)} \oplus 0 \right) \in E,\\
&\left( f^{\odot (-1)} \odot g \oplus 0, 0 \right) = \left( f^{\odot (-1)} \odot g \oplus 0, f^{\odot (-1)} \odot f \oplus 0 \right) \in E.
\end{align*}

Conversely we assume that $\left( f \odot g^{\odot (-1)} \oplus 0, 0 \right), \left( f^{\odot (-1)} \odot g \oplus 0, 0 \right) \in E$ hold.
Because
\begin{align*}
(f \oplus g, g) &= \left( \left( f \odot g^{\odot (-1)} \oplus 0 \right) \odot g, 0 \odot g \right) \in E,\\
(f \oplus g, f) &= \left( \left( f^{\odot (-1)} \odot g \oplus 0 \right) \odot f, 0 \odot f \right) \in E,
\end{align*}
we have $(f, g) \in E$ by symmetry and transitivity.
\end{proof}

\begin{lemma}[cf.~{\cite[Lemma 3.4]{JuAe4}}]
	\label{lem3}
For any semiring $S$, the semiring $S^2$ is generated by $(1, 0)$.
\end{lemma}

\begin{proof}
For any $x \in S$, we have $(x, 0) = (x \cdot 1, x \cdot 0) \in \langle (1, 0) \rangle$.
Thus, by symmetry and transitivity, for any $x, y \in S$, $(x, y) \in \langle (1, 0) \rangle$ holds, which means that $S^2 = \langle (1, 0) \rangle$.
\end{proof}

\begin{prop}[cf.~{\cite[Proposition 3.10]{Joo=Mincheva2}}]
	\label{prop1}
For a congruence $E$ on $\overline{\boldsymbol{T}(X_1, \ldots, X_n)}$, $E$ is finitely generated as a congruence if and only if there exists $f \in \overline{\boldsymbol{T}(X_1, \ldots, X_n)}$ such that $E = \langle (f, 0) \rangle$.
In particular, if $E$ is proper, then $f$ above can be taken as $f \ge 0$.
\end{prop}

\begin{proof}
The if part is clear.
We shall show the only if part.
If $E$ is trivial, the assertion is clear, and if $E$ is improper, we have the conclusion by Lemma \ref{lem3}.
Assume that $E$ is nontrivial and proper.
Assume that $E$ is finitely generated as a congruence.
Then there exist $(f_1, g_1), \ldots, (f_m, g_m) \in \overline{\boldsymbol{T}(X_1, \ldots, X_n)}^2$ that generate $E$, i.e., $\langle (f_1, g_1), \ldots, (f_m, g_m) \rangle = E$.
If there exists $j$ such that $f_j \not= - \infty$ and $g_j = -\infty$, then, since the inverse $f_j^{\odot (-1)}$ exists, we have $(0, -\infty) = \left( f_j \odot f_j^{\odot (-1)}, g_j \odot f_j^{\odot (-1)} \right) \in E$, which is a contradiction by Lemma \ref{lem3}.
Hence, without loss of generality, we can assume that for any $i$, $(f_i, g_i) \not\in \Delta$ and $f_i, g_i \not= -\infty$.
By Lemma \ref{lem2},
\begin{align*}
&\left( f_1 \odot g_1^{\odot (-1)} \oplus 0, 0 \right), \left( f_1^{\odot (-1)} \odot g_1 \oplus 0, 0 \right),\\
&\ldots, \left( f_m \odot g_m^{\odot (-1)} \oplus 0, 0 \right), \left( f_m^{\odot (-1)} \odot g_m \oplus 0, 0 \right)
\end{align*}
are also generators of $E$.
Let $h_i := f_i \odot g_i^{\odot (-1)} \oplus 0$ and $h^{\prime}_i := f_i^{\odot (-1)} \odot g_i \oplus 0$ for each $i$.
Then $h_i \ge 0$ and $h^{\prime}_i \ge 0$, and hence we have
\begin{align*}
h_1 \odot \cdots \odot h_m \odot h_1^{\prime} \odot \cdots \odot h_m^{\prime} \ge h_1 \oplus \cdots \oplus h_m \oplus h_1^{\prime} \oplus \cdots \oplus h_m^{\prime} \ge h_i, h_i^{\prime} \ge 0
\end{align*}
for any $i$.
By \cite[Proposition 2.2(iii)]{Joo=Mincheva1},
\begin{align*}
(h_i, 0), (h_i^{\prime}, 0) &\in \langle (h_1 \oplus \cdots \oplus h_m \oplus h_1^{\prime} \oplus \cdots \oplus h_m^{\prime}, 0) \rangle\\
&\subset \langle (h_1 \odot \cdots \odot h_m \odot h_1^{\prime} \odot \cdots \odot h_m^{\prime}, 0) \rangle
\end{align*}
hold for any $i$.
Since $(h_1, 0), \ldots, (h_m, 0), (h_1^{\prime}, 0), \ldots, (h_m^{\prime}, 0)$ are generators of $E$, the converse inclusions are clear.
Hence
\begin{align*}
E &= \langle (h_1 \odot \cdots \odot h_m \odot h_1^{\prime} \odot \cdots \odot h_m^{\prime}, 0) \rangle\\
&= \langle (h_1 \oplus \cdots \oplus h_m \oplus h_1^{\prime} \oplus \cdots \oplus h_m^{\prime}, 0) \rangle
\end{align*}
hold.
\end{proof}

By Lemma \ref{lem1} and Proposition \ref{prop1}, we have the following corollary:

\begin{cor}
	\label{cor4}
For a finitely generated congruence $E$ on $\overline{\boldsymbol{T}(X_1, \ldots, X_n)}$, the congruence variety $\boldsymbol{V}(E)$ is a finite union of $\boldsymbol{R}$-rational polyhedral sets.
\end{cor}

Next, we prove that the congruence on $\overline{\boldsymbol{T}(X_1, \ldots, X_n)}$ defined by a finite union of $\boldsymbol{R}$-rational polyhedral sets is finitely generated as a congruence.

\begin{lemma}
	\label{lem4}
Let $V$ be a nonempty subset of $\boldsymbol{R}^n$.
For any $(g, 0) \in \boldsymbol{E}\left(\overline{V}\right)$ such that $g \ge 0$, there exists $N \in \boldsymbol{Z}_{>0}$ such that $\operatorname{dist}\left(x, \overline{V}\right)^{\odot N} \ge g(x)$ holds for any $x \in \boldsymbol{R}^n$.
\end{lemma}

\begin{proof}
As $(g, 0) \in \boldsymbol{E} \left( \overline{V} \right)$, for any $x \in \overline{V}$, we have $g(x) = 0$.
Let $x \in \boldsymbol{R}^n \setminus \overline{V}$ and $y$ be a point in $\overline{V}$ satisfying $\operatorname{dist}(x, y) = \operatorname{dist} \left( x, \overline{V} \right)$.
Let $z := x - y$.
Then we have $\operatorname{dist} \left( x, \overline{V} \right) = \sqrt{(z)_1^2 + \cdots + (z)_n^2} \ge |(z)_1|, \ldots, |(z)_n|$.
Hence $\operatorname{dist} \left(x, \overline{V} \right)^{\odot n} \ge |(z)_1| \odot \cdots \odot |(z)_n|$ holds.
Because $\overline{V}$ is nonempty and $(g, 0) \in \boldsymbol{E} \left( \overline{V} \right)$, $g$ is not $-\infty$.
Thus there exist $g_1, g_2 \in \overline{\boldsymbol{T}[X_1, \ldots, X_n]} \setminus \{ -\infty \}$ such that $g = g_1 \odot g_2^{\odot (-1)}$.
Let $g_i^{\prime} = \bigoplus_{\boldsymbol{j}_i = (j_{i, 1}, \ldots, j_{i, n}) \in B_i} b_{\boldsymbol{j}_i} \odot X_1^{\odot j_{i, 1}} \odot \cdots \odot X_n^{\odot j_{i, n}} \in \boldsymbol{T}[X_1, \ldots, X_n] \setminus \{ -\infty \}$ be a representative of $g_i$, where $B_i$ is a finite subset of $\left( \boldsymbol{Z}_{\ge 0} \right)^n$ and $b_{\boldsymbol{j}_i} \in \boldsymbol{R}$ for $\boldsymbol{j}_i = (j_{i, 1}, \ldots, j_{i, n}) \in B_i$.
Let $g_1(x) = b_{\boldsymbol{j}_1} \odot (x)_1^{\odot j_{1, 1}} \odot \cdots \odot (x)_n^{\odot j_{1, n}}$ for some $\boldsymbol{j}_1 = (j_{1, 1}, \ldots, j_{1, n}) \in B_1$ and $g_2(y) = b_{\boldsymbol{j}_2} \odot (y)_1^{\odot j_{2, 1}} \odot \cdots \odot (y)_n^{\odot j_{2, n}}$ for some $\boldsymbol{j}_2 = (j_{2, 1}, \ldots, j_{2, n}) \in B_2$.
We have
\begin{align*}
g_1(x) = g_1(y + z) &= b_{\boldsymbol{j}_1} \odot (y + z)_1^{\odot j_{1, 1}} \odot \cdots \odot (y + z)_n^{\odot j_{1, n}}\\
&= b_{\boldsymbol{j}_1} \odot (y)_1^{\odot j_{1, 1}} \odot \cdots \odot (y)_n^{\odot j_{1, n}} \odot (z)_1^{\odot j_{1, 1}} \odot \cdots \odot (z)_n^{\odot j_{1, n}}\\
&\le g_1(y) \odot (z)_1^{\odot j_{1, 1}} \odot \cdots \odot (z)_n^{\odot j_{1, n}}
\end{align*}
and
\begin{align*}
g_2(y) &= b_{\boldsymbol{j}_2} \odot (y)_1^{\odot j_{2, 1}} \odot \cdots \odot (y)_n^{\odot j_{2, n}}\\
&= b_{\boldsymbol{j}_2} \odot (y + z - z)_1^{\odot j_{2, 1}} \odot \cdots \odot (y + z - z)_n^{\odot j_{2, n}}\\
&= b_{\boldsymbol{j}_2} \odot (y + z)_1^{\odot j_{2, 1}} \odot \cdots \odot (y + z)_n^{\odot j_{2, n}} \odot (- z)_1^{\odot j_{2, 1}} \odot \cdots \odot (- z)_n^{\odot j_{2, n}}\\
&\le g_2(y + z) \odot (- z)_1^{\odot j_{2, 1}} \odot \cdots \odot (- z)_n^{\odot j_{2, n}}.
\end{align*}
Since $0 = g(y) = g_1(y) \odot g_2(y)^{\odot (-1)}$, we have $g_1(y) = g_2(y) \in \boldsymbol{R}$.
Therefore 
\begin{align*}
g(x) = g(y + z) &= g_1(y + z) \odot g_2(y + z)^{\odot (-1)}\\
&\le  g_1(y) \odot (z)_1^{\odot j_{1, 1}} \odot \cdots \odot (z)_n^{\odot j_{1, n}} \odot g_2(y + z)^{\odot (-1)}\\
&= g_2(y) \odot (z)_1^{\odot j_{1, 1}} \odot \cdots \odot (z)_n^{\odot j_{1, n}} \odot g_2(y + z)^{\odot (-1)}\\
&\le  (z)_1^{\odot j_{1, 1}} \odot \cdots \odot (z)_n^{\odot j_{1, n}} \odot (- z)_1^{\odot j_{2, 1}} \odot \cdots \odot (- z)_n^{\odot j_{2, n}}\\
&=  (z)_1^{\odot (j_{1, 1} - j_{2, 1})} \odot \cdots \odot (z)_n^{\odot (j_{1, n} - j_{2, n})}\\
&\le \operatorname{max} \left\{ (z)_1^{\odot (j_{1, 1}^{\prime} - j_{2, 1}^{\prime})} \odot \cdots \odot (z)_n^{\odot (j_{1, n}^{\prime} - j_{2, n}^{\prime})} \, \middle| \, \begin{alignedat}{2}(j_{1, 1}^{\prime}, \ldots, j_{1, n}^{\prime}) \in B_1 \\ (j_{2, 1}^{\prime}, \ldots, j_{2, n}^{\prime}) \in B_2 \end{alignedat} \right\}
\end{align*}
hold.
The inequality $\operatorname{dist} \left( x, \overline{V} \right)^{\odot nk} \ge g(x)$ holds for $k \in \boldsymbol{Z}_{>0}$ greater than $\operatorname{max}\left\{ |j_{1, 1}^{\prime} - j_{2, 1}^{\prime}|, \ldots, |j_{1, n}^{\prime} - j_{2, n}^{\prime}| \,|\, (j_{1, 1}^{\prime}, \ldots, j_{1, n}^{\prime}) \in B_1, (j_{2, 1}^{\prime}, \ldots, j_{2, n}^{\prime}) \in B_2 \right\}$.
We note that this $k$ is independent of the choice of $x$.
In conclusion, for $N = nk$, we have $\operatorname{dist} \left(x, \overline{V} \right)^{\odot N} \ge g(x)$ for any $x \in \boldsymbol{R}^n$.
\end{proof}

\begin{cor}
	\label{cor5}
Let $V$ be a nonempty subset of $\boldsymbol{R}^n$.
For $f \in \overline{\boldsymbol{T}(X_1, \ldots, X_n)}$, assume that $(1)$ $\boldsymbol{V}((f, 0)) = \overline{V}$ holds and $(2)$ there exists $k^{\prime} \in \boldsymbol{Z}_{>0}$ such that for any $x \in \boldsymbol{R}^n$, $f(x)^{\odot k^{\prime}} \ge \operatorname{dist} \left( x, \overline{V} \right)$ holds.
Then for any $(g, 0) \in \boldsymbol{E}\left( \overline{V} \right)$ such that $g \ge 0$, there exists $N^{\prime} \in \boldsymbol{Z}_{>0}$ such that $f^{\odot N^{\prime}} \ge g$ and $\boldsymbol{E} \left( \overline{V} \right) = \langle (f, 0) \rangle$ holds.
\end{cor}

\begin{proof}
By $(1)$, we have $\boldsymbol{E}\left(\overline{V} \right) \supset \langle (f, 0) \rangle$.
For $N$ in Lemma \ref{lem4}, by $(2)$, we have $0 \le g(x) \le \operatorname{dist} \left( x, \overline{V} \right)^{\odot N} \le f(x)^{\odot Nk^{\prime}}$ for any $x \in \boldsymbol{R}^n$, thus $0 \le g \le f^{\odot N^{\prime}}$ for $N^{\prime} = Nk^{\prime}$.
By \cite[Proposition 2.2(iii)]{Joo=Mincheva1}, this means that $(g, 0) \in \langle (f, 0) \rangle$.
Since $V$ is nonempty, by \cite[Lemmas 3.4 and 3.7]{JuAe4}, if $(h, -\infty) \in \boldsymbol{E} \left( \overline{V} \right)$, then $h = -\infty$, and thus we have $\boldsymbol{E} \left( \overline{V} \right) \subset \langle (f, 0) \rangle$ by Lemma \ref{lem2}.
\end{proof}

\begin{cor}
	\label{cor6}
Let $f \in \overline{\boldsymbol{T}(X_1, \ldots, X_n)}$ be represented by $a \odot X_1^{\odot i_1} \odot \cdots \odot X_n^{\odot i_n} \oplus 0 \in \boldsymbol{T}\left[X_1^{\pm 1}, \ldots, X_n^{\pm 1}\right]$.
If $a \in \boldsymbol{R}$ and at least one of $i_1, \ldots, i_n$ is not zero, then $f(x) \ge \operatorname{dist}(x, \boldsymbol{V}((f, 0)))$ for any $x \in \boldsymbol{R}^n$ and $\boldsymbol{E}(\boldsymbol{V}((f, 0))) = \langle (f, 0) \rangle$ hold.
\end{cor}

\begin{proof}
By assumption, $\sqrt{i_1^2 + \cdots + i_n^2} \ge 1$ holds.
By the definition of $f$, $\boldsymbol{V}((f, 0))$ is nonempty, and for any $x \in \boldsymbol{R}^n$, $f(x)$ is $\operatorname{dist}(x, \boldsymbol{V}((f, 0)))$ multiplied (in the usual meaning) by the positive real number $\sqrt{i_1^2 + \cdots + i_n^2}$.
Hence $f(x) \ge \operatorname{dist}(x, \boldsymbol{V}((f, 0)))$ holds.
By Corollary \ref{cor5}, we have $\boldsymbol{E}(\boldsymbol{V}((f, 0))) = \langle (f, 0) \rangle$.
\end{proof}

\begin{lemma}
	\label{lem5}
Let $m \in \boldsymbol{Z}_{\ge 1}$.
For $j = 1, \ldots, m$, let $f_j \in \overline{\boldsymbol{T}(X_1, \ldots, X_n)}$ be represented by $a_j \odot X_1^{\odot i_{j, 1}} \odot \cdots \odot X_n^{\odot i_{j, n}} \oplus 0 \in \boldsymbol{T}\left[X_1^{\pm 1}, \ldots, X_n^{\pm 1}\right]$.
If $a_j \in \boldsymbol{R}$ and at least one of $i_{j, 1}, \ldots, i_{j, n}$ is not zero for any $j$ and $\bigcap_{j = 1}^m \boldsymbol{V}((f_j, 0))$ is nonempty, then the following hold:

$(1)$ $\boldsymbol{V}((f_1 \odot \cdots \odot f_m, 0)) = \boldsymbol{V}((f_1 \oplus \cdots \oplus f_m, 0)) = \bigcap_{j = 1}^m \boldsymbol{V}((f_j, 0))$,

$(2)$ there exists $k \in \boldsymbol{Z}_{>0}$ such that $\left(f_1(x) \odot \cdots \odot f_m(x) \right)^{\odot k} \ge \operatorname{dist} \left(x, \bigcap_{j = 1}^m \boldsymbol{V}((f_j, 0))\right)$, $\left( f_1(x) \oplus \cdots \oplus f_m(x) \right)^{\odot k} \ge \operatorname{dist}\left(x, \bigcap_{j = 1}^m \boldsymbol{V}((f_j, 0))\right)$ hold for any $x \in \boldsymbol{R}^n$, and

$(3)$ $\boldsymbol{E} \left( \bigcap_{j = 1}^m \boldsymbol{V}((f_j, 0))\right) = \langle (f_1 \odot \cdots \odot f_m, 0) \rangle = \langle (f_1 \oplus \cdots \oplus f_m, 0) \rangle$.
\end{lemma}

\begin{proof}
$(1)$ and $(2)$ imply $(3)$ by Corollary \ref{cor5}.

Since $f_j \ge 0$, $f_1 \odot \cdots \odot f_m \ge f_1 \oplus \cdots \oplus f_m \ge 0$ hold, and thus $\langle (f_1 \odot \cdots \odot f_m, 0) \rangle \supset \langle (f_1 \oplus \cdots \oplus f_m, 0) \rangle$ holds by \cite[Proposition 2.2(iii)]{Joo=Mincheva1}.

Conversely, as $f_1 \oplus \cdots \oplus f_m \ge f_j$, we have $(f_1 \oplus \cdots \oplus f_m)^{\odot m} \ge f_1 \odot \cdots \odot f_m$, and hence $\langle (f_1 \odot \cdots \odot f_m, 0) \rangle \subset \langle (f_1 \oplus \cdots \oplus f_m, 0) \rangle$ by \cite[Proposition 2.2(iii)]{Joo=Mincheva1}.

Because for any $x \in \bigcap_{j = 1}^m \boldsymbol{V}((f_j, 0))$, $f_1(x) = \cdots = f_m(x) = 0$ hold, we have $\boldsymbol{V}((f_1 \odot \cdots \odot f_m, 0)) \supset \bigcap_{j = 1}^m \boldsymbol{V}((f_j, 0))$.

Conversely, for any $x \in \boldsymbol{V}((f_1 \odot \cdots \odot f_m, 0))$, since $f_i \ge 0$ and $f_1(x) \odot \cdots \odot f_m(x) = 0$, we have $f_1(x) = \cdots = f_m(x) = 0$ hold.
Thus $\boldsymbol{V}((f_1 \odot \cdots \odot f_m, 0)) \subset \bigcap_{j = 1}^m \boldsymbol{V}((f_j, 0))$.

$(2)$ follows from Corollary \ref{cor:decomposition3} and Lemma \ref{lem:distance}.
\end{proof}

\begin{lemma}
	\label{lem6}
Let $V$ and $W$ be nonempty subsets of $\boldsymbol{R}^n$.
Assume $f, g \in \overline{\boldsymbol{T}(X_1, \ldots, X_n)}$ satisfy the following conditions: $\boldsymbol{V}((f, 0)) = \overline{V}$, $\boldsymbol{V}((g, 0)) = \overline{W}$, and there exist $k_1, k_2 \in \boldsymbol{Z}_{>0}$ such that $f(x)^{\odot k_1} \ge \operatorname{dist} \left(x, \overline{V} \right)$ and $g(x)^{\odot k_2} \ge \operatorname{dist} \left(x, \overline{W} \right)$ hold for any $x \in \boldsymbol{R}^n$.
Then the following hold:

$(1)$ $\boldsymbol{V}\left( \left( \left(f^{\odot (-1)} \oplus g^{\odot (-1)} \right)^{\odot (-1)}, 0 \right) \right) = \overline{V} \cup \overline{W}$,

$(2)$ for $k := \operatorname{max}\{ k_1, k_2 \}$, $\left(f(x)^{\odot (-1)} \oplus g(x)^{\odot (-1)}\right)^{\odot (-k)} \ge \operatorname{dist} \left(x, \overline{V} \cup \overline{W} \right)$ holds for any $x \in \boldsymbol{R}^n$, and

$(3)$ $\boldsymbol{E} \left(\overline{V} \cup \overline{W} \right) = \left\langle \left( \left(f^{\odot (-1)} \oplus g^{\odot (-1)} \right)^{\odot (-1)}, 0 \right) \right\rangle$.
\end{lemma}

\begin{proof}
$(1)$ and $(2)$ imply $(3)$ by Corollary \ref{cor5}.

For any $x \in \overline{V}$ and $y \in \overline{W}$, we have $f(x) = 0$ and $g(y) = 0$.
Thus for any $z \in \overline{V} \cup \overline{W}$, at least one of $f(z)$ and $g(z)$ is zero.
By assumption, $f \ge 0$ and $g \ge 0$ hold, and hence we have $f^{\odot (-1)} \le 0$ and $g^{\odot (-1)} \le 0$.
From this, $f(z)^{\odot (-1)} \oplus g(z)^{\odot (-1)} = 0$ holds.
Thus we have $\boldsymbol{V} \left( \left( \left(f^{\odot (-1)} \oplus g^{\odot (-1)} \right)^{\odot (-1)}, 0 \right)\right) \supset \overline{V} \cup \overline{W}$ and we also have $\left(\left(f^{\odot (-1)} \oplus g^{\odot (-1)} \right)^{\odot (-1)}, 0 \right) \in \boldsymbol{E}\left(\overline{V} \cup \overline{W}\right)$, and thus $\left\langle \left(\left(f^{\odot (-1)} \oplus g^{\odot (-1)}\right)^{\odot (-1)}, 0\right) \right\rangle \subset \boldsymbol{E}\left(\overline{V} \cup \overline{W}\right)$.
Since $\left( f(x)^{\odot (-1)} \oplus g(x)^{\odot (-1)} \right)^{\odot (-1)} = 0$ implies $f(x) = 0$ or $g(x) = 0$, we have $\boldsymbol{V}\left(\left(\left(f^{\odot (-1)} \oplus g^{\odot (-1)}\right)^{\odot (-1)}, 0\right)\right) \subset \overline{V} \cup \overline{W}$.

Because $\overline{V} \cup \overline{W} \supset \overline{V}, \overline{W}$, we have $\operatorname{dist}\left(x, \overline{V} \cup \overline{W} \right) \le \operatorname{dist}\left(x, \overline{V}\right), \operatorname{dist}\left(x, \overline{W} \right)$.
Hence for any $x \in \boldsymbol{R}^n$, if $f(x) \ge g(x)$, then
\begin{align*}
\left( f(x)^{\odot (-1)} \oplus g(x)^{\odot (-1)} \right)^{\odot (-k)} = g(x)^{\odot k} \ge \operatorname{dist}\left(x, \overline{W} \right) \ge \operatorname{dist}\left(x, \overline{V} \cup \overline{W}\right)
\end{align*}
hold; if $f(x) \le g(x)$, then
\begin{align*}
\left( f(x)^{\odot (-1)} \oplus g(x)^{\odot (-1)} \right)^{\odot (-k)} = f(x)^{\odot k} \ge \operatorname{dist}\left(x, \overline{V}\right) \ge \operatorname{dist}\left(x, \overline{V} \cup \overline{W}\right)
\end{align*}
hold.
\end{proof}

Since $\boldsymbol{E}(\varnothing) = \overline{\boldsymbol{T}(X_1, \ldots, X_n)}^2$, by Corollary \ref{cor4} and Lemmas \ref{lem3}, \ref{lem5}, \ref{lem6}, we have Theorem \ref{thm:main}.
Also Corollary \ref{cor1} follows from Corollary \ref{cor4} and Theorem \ref{thm:main}.

\begin{lemma}
	\label{lem7}
For $f_1, \ldots, f_n \in \overline{\boldsymbol{T}(Y_1, \ldots, Y_m)}\setminus \{ -\infty \}$, let $\theta : \boldsymbol{R}^m \to \boldsymbol{R}^n; y \mapsto (f_1(y), \ldots, f_n(y))$.
Then $\theta$ maps a $\boldsymbol{R}$-rational polyhedral set in $\boldsymbol{R}^m$ to a finite union of $\boldsymbol{R}$-rational polyhedral sets in $\boldsymbol{R}^n$.
\end{lemma}

\begin{proof}
Each $f_i$ breaks $\boldsymbol{R}^m$ into finitely many closed subsets on each which $f_i$ is integral affine.
Since the number of $f_i$ is finite, we can break $\boldsymbol{R}^m$ into finitely many closed subsets on each of which each $f_i$ is integral affine.
As such each closed sunset is also a $\boldsymbol{R}$-rational polyhedral set, its image by $\theta$ is also a $\boldsymbol{R}$-rational polyhedral set, and hence we have the conclusion.
\end{proof}

\begin{cor}
	\label{cor7}
Let $V$ (resp. $W$) be a subset of $\boldsymbol{R}^n$ (resp. $\boldsymbol{R}^m$).
Let $\psi : \overline{\boldsymbol{T}(X_1, \ldots, X_n)} / \boldsymbol{E}(V) \to \overline{\boldsymbol{T}(Y_1, \ldots, Y_m)} / \boldsymbol{E}(W)$ be a $\boldsymbol{T}$-algebra homomorphism and $\pi_{\boldsymbol{E}(V)} : \overline{\boldsymbol{T}(X_1, \ldots, X_n)} \twoheadrightarrow \overline{\boldsymbol{T}(X_1, \ldots, X_n)} / \boldsymbol{E}(V)$ and $\pi_{\boldsymbol{E}(W)} : \overline{\boldsymbol{T}(Y_1, \ldots, Y_m)} \twoheadrightarrow \overline{\boldsymbol{T}(Y_1, \ldots, Y_m)} / \boldsymbol{E}(W)$ the natural surjective $\boldsymbol{T}$-algebra homomorphisms.
Fix an element $f_i \in \pi_{\boldsymbol{E}(W)}^{-1}\left(\psi \left(\pi_{\boldsymbol{E}(V)}(X_i) \right)\right) \setminus \{ -\infty \}$ for any $i$.
Let $\theta : \boldsymbol{R}^m \to \boldsymbol{R}^n; y \mapsto (f_1(y), \ldots, g_n(y))$. 
If $\psi$ is injective, $\theta(\overline{W})$ is closed and $\boldsymbol{E}(W)$ is finitely generated as a congruence, then $\boldsymbol{E}(V)$ is finitely generated as a congruence.
In particular, when $\overline{\boldsymbol{T}(X_1, \ldots, X_n)} / \boldsymbol{E}(V)$ is isomorphic to $\overline{\boldsymbol{T}(Y_1, \ldots, Y_m)} / \boldsymbol{E}(W)$ as a $\boldsymbol{T}$-algebra, $\boldsymbol{E}(V)$ is finitely generated as a congruence if and only if so is $\boldsymbol{E}(W)$.
\end{cor}

\begin{proof}
There exist finitely many closed half-spaces $H_{11}, \ldots, H_{1k_1}, H_{21}, \ldots, H_{jk_j}$ in $\boldsymbol{R}^n$ that have integral normal vectors, such that $\bigcup_{a = 1}^j \bigcap_{b = 1}^{k_a} H_{ab} = \overline{W}$ by Corollary \ref{cor4}.
By \cite[Theorem 3.15]{JuAe4}, we have
\begin{align*}
\overline{V} = \theta \left( \overline{W} \right) = \theta \left(\bigcup_{a = 1}^j \bigcap_{b = 1}^{k_a} H_{ab} \right) = \bigcup_{a = 1}^j \theta \left( \bigcap_{b = 1}^{k_a} H_{ab} \right).
\end{align*}
By Lemma \ref{lem7} and Theorem \ref{thm:main}, we have the conclusion.
\end{proof}

\begin{proof}[Proof of Corollary \ref{cor3}]
If $V$ is empty, then we have
\begin{align*}
\langle (0, -\infty) \rangle_{\overline{\boldsymbol{T}[X_1, \ldots, X_n]}} = \overline{\boldsymbol{T}[X_1, \ldots, X_n]}^2 = \boldsymbol{E}(\varnothing)_1
\end{align*}
by Lemma \ref{lem3}.
For any $(g, h) \in \overline{\boldsymbol{T}[X_1, \ldots, X_n]}^2$, we have $(-\infty, -\infty) = (g \odot (-\infty), h \odot (-\infty)) \in \Delta \subset \langle (0, -\infty) \rangle_{\overline{\boldsymbol{T}[X_1, \ldots, X_n]}}$, and hence
\begin{align*}
\boldsymbol{E}(\varnothing)_1 =&~ \left\{ (g, h) \in \overline{\boldsymbol{T}[X_1, \ldots, X_n]}^2 \,\middle|\,\right.\\
&~\left.\exists k \in \boldsymbol{Z}_{>0}: \left(g \odot (-\infty)^{\odot k}, h \odot (-\infty)^{\odot k} \right) \in \langle (0, -\infty) \rangle_{\overline{\boldsymbol{T}[X_1, \ldots, X_n]}} \right\}.
\end{align*}
Cleary $0 \ge -\infty$ on $\boldsymbol{T}^n$.

Assume that $V$ is nonempty.
Then, by Lemmas \ref{lem2}, \ref{lem4}, \ref{lem5} and \ref{lem6}, there exists $f \in \overline{\boldsymbol{T}(X_1, \ldots, X_n)} \setminus \{ -\infty \}$ such that $\boldsymbol{E}(V) = \langle (f, 0) \rangle$, $f \ge 0$ on $\boldsymbol{R}^n$ and for any $(g, h) \in \boldsymbol{E}(V) \setminus \{ (-\infty, -\infty) \}$, there exists $k \in \boldsymbol{Z}_{> 0}$ such that
\begin{align*}
0 \le g \odot h^{\odot (-1)} \oplus 0 \le f^{\odot k}
\end{align*}
and
\begin{align*}
0 \le g^{\odot (-1)} \odot h \oplus 0 \le f^{\odot k}
\end{align*}
on $\boldsymbol{R}^n$.
There exist $f_1, f_2 \in \overline{\boldsymbol{T}[X_1, \ldots, X_n]} \setminus \{ -\infty \}$ such that $f = f_1 \odot f_2^{\odot (-1)}$.
Therefore we have
\begin{align*}
h \odot f_2^{\odot k} \le (g \oplus h) \odot f_2^{\odot k} \le h \odot f_1^{\odot k}
\end{align*}
and
\begin{align*}
g \odot f_2^{\odot k} \le (g \oplus h) \odot f_2^{\odot k} \le g \odot f_1^{\odot k}
\end{align*}
on $\boldsymbol{R}^n$.
Note that, by \cite[Lemma 3.7.4]{Giansiracusa=Giansiracusa2}, these inequalities are also true on $\boldsymbol{T}^n$ if all these are tropical polynomial functions.
Thus, since
\begin{align*}
\left( h \odot f_2^{\odot k}, h \odot f_1^{\odot k} \right), \left( g \odot f_2^{\odot k}, g \odot f_1^{\odot k} \right) \in \langle (f_1, f_2) \rangle_{\overline{\boldsymbol{T}(X_1, \ldots, X_n)}},
\end{align*}
we have
\begin{align*}
\left( h \odot f_2^{\odot k}, (g \oplus h) \odot f_2^{\odot k} \right), \left( g \odot f_2^{\odot k}, (g \oplus h) \odot f_2^{\odot k} \right) \in \langle (f_1, f_2) \rangle_{\overline{\boldsymbol{T}(X_1, \ldots, X_n)}}
\end{align*}
by \cite[Proposition 2.2(iii)]{Joo=Mincheva1}.
This concludes that $\left( g \odot f_2^{\odot k}, h \odot f_2^{\odot k} \right) \in \langle (f_1, f_2) \rangle_{\overline{\boldsymbol{T}(X_1, \ldots, X_n)}}$.
Note that if both $g$ and $h$ are tropical polynomial functions, we have also $\left( g \odot f_2^{\odot k}, h \odot f_2^{\odot k} \right) \in \langle (f_1, f_2) \rangle_{\overline{\boldsymbol{T}[X_1, \ldots, X_n]}}$.

Let $(g, h) \in \boldsymbol{E}(V)_1$.
By definition, we can regard that $(g, h) \in \boldsymbol{E}(V)$, and thus there exists $k \in \boldsymbol{Z}_{>0}$ such that $\left( g \odot f_2^{\odot k}, h \odot f_2^{\odot k} \right) \in \langle (f_1, f_2) \rangle_{\overline{\boldsymbol{T}[X_1, \ldots, X_n]}}$ by the argument above if $(g, h) \not= (-\infty, -\infty)$.
Clearly $(-\infty, -\infty) = (-\infty \odot f_2, -\infty \odot f_2) \in \langle (f_1, f_2) \rangle_{\overline{\boldsymbol{T}[X_1, \ldots, X_n]}}$.

For $(g, h) \in \overline{\boldsymbol{T}[X_1, \ldots, X_n]}^2$, we assume that there exists $k \in \boldsymbol{Z}_{>0}$ such that $\left( g \odot f_2^{\odot k}, h \odot f_2^{\odot k} \right) \in \langle (f_1, f_2) \rangle_{\overline{\boldsymbol{T}[X_1, \ldots, X_n]}}$.
By \cite[Corollary 2.12]{Bertram=Easton}, for any $x \in V$, we have $g(x) \odot f_2(x)^{\odot k} = h(x) \odot f_2(x)^{\odot k}$.
As $f_2(x) \in \boldsymbol{R}$, we have $g(x) = h(x)$, and thus $(g, h) \in \boldsymbol{E}(V)_1$.

Finally because $f \ge 0$ on $\boldsymbol{R}^n$, we ahve $f_1 \ge f_2$ on $\boldsymbol{R}^n$.
By \cite[Lemma 3.7.4]{Giansiracusa=Giansiracusa2} again, this inequality also holds on $\boldsymbol{T}^n$.
\end{proof}

Now we can characterize rational function semifields of tropical curves.
By Theorem \ref{thm:main}, the following lemma is clear:

\begin{lemma}
	\label{lem8}
Let $(x_1, \ldots, x_n)$ be a point of $\boldsymbol{R}^n$ and $(s_1, \ldots, s_n) \in \boldsymbol{Z}^n$ a vector other than $(0, \ldots, 0)$.
For $l \in \boldsymbol{R}_{\ge 0} \cup \{ \infty \}$ and the interval $e_l := \overline{\{ (x_1 + t s_1, \ldots, x_n + t s_n) \,|\, t \in [0, l) \}}$, the congruence $\boldsymbol{E}(e_l)$ is finitely generated as a congruence.
\end{lemma}

Let $\Gamma$ be a tropical curve.
By \cite[Theorem 1.1]{JuAe2}, $\operatorname{Rat}(\Gamma)$ is finitely generated as a semifield over $\boldsymbol{T}$.
For any generators $f_1, \ldots, f_n$ other than $-\infty$ of $\operatorname{Rat}(\Gamma)$ as a semifield over $\boldsymbol{T}$, by \cite[Lemma 3.11]{JuAe4}, the correspondence $X_i \mapsto f_i$ induces the surjective $\boldsymbol{T}$-algebra homomorphism $\psi : \overline{\boldsymbol{T}(X_1, \ldots, X_n)} \twoheadrightarrow \operatorname{Rat}(\Gamma)$ mapping each $X_i$ to $f_i$.

\begin{cor}
	\label{cor8}
In the above setting, the kernel congruence $\operatorname{Ker}(\psi)$ is finitely generated as a congruence.
\end{cor}

\begin{proof}
By \cite[Proposition 3.12(5)]{JuAe4}, we have the isometry $\theta : \Gamma \setminus \Gamma_{\infty} \to \boldsymbol{V}(\operatorname{Ker}(\psi)) ; x \mapsto (f_1(x), \ldots, f_n(x))$.
Note that $\boldsymbol{V}(\operatorname{Ker}(\psi))$ has the metric defined by the lattice length.
Since $f_1, \ldots, f_n$ are rational functions on $\Gamma$, we can break $\Gamma$ into finitely many intervals $e_1, \ldots, e_m$ each which has a finite endpoint $x_j$ and on each which every $f_i$ has a constant slope $s_{ji}$ from $x_j$ to the other endpoint of $e_j$.
Hence the image $\theta(e_j)$ is also an interval and is written as $\overline{\{ (f_1(x_j) + ts_{j1}, \ldots, f_n(x_j) + ts_{jn}) \,|\, t \in [0, l(e_j)) \}}$, where $l(e_j)$ denotes the length of $e_j$.
By Lemma \ref{lem8} and Theorem \ref{thm:main}, we have the conclusion.
\end{proof}

Let $V$ be a closed connected subset of $\boldsymbol{R}^n$ of geometric dimension one whose associated congruence $\boldsymbol{E}(V)$ is finitely generated as a congruence.
By Corollary \ref{cor4}, $V = \boldsymbol{V}(\boldsymbol{E}(V))$ has the metric defined by the lattice length.

\begin{lemma}
	\label{lem9}
If $n \ge 2$ and there exist no two distinct rays of $V$ whose primitive direction verctors coincide, then for any ray $e$ in $V$, there exist a point $x$ of $e$ and distinct vectors $\boldsymbol{s}_1, \ldots, \boldsymbol{s}_m \in \boldsymbol{R}^n \setminus \{(0, \ldots, 0)\}$ and $f \in \overline{\boldsymbol{T}(X_1, \ldots, X_n)}$ satisfying the following five conditions:

$(1)$ the cone $\{ x \} + \operatorname{Cone}\{ \boldsymbol{s}_1, \ldots, \boldsymbol{s}_m \}$ with the vertex $x$ is of geometric dimension $n$,

$(2)$ the relative interior $\operatorname{Int}(\{ x \} + \operatorname{Cone}\{ \boldsymbol{s}_1, \ldots, \boldsymbol{s}_m \} )$ contains $e_x := \{ x + t\boldsymbol{l} \,|\, t \in \boldsymbol{R}_{\ge 0} \}$, where $\boldsymbol{l}$ is the primitive direction vector of $e$ such that $e_x$ is contained in $e$,

$(3)$ the intersection $(\{ x \} + \operatorname{Cone}\{ \boldsymbol{s}_1, \ldots, \boldsymbol{s}_m \}) \cap (V \setminus e_x)$ is empty, 

$(4)$ $f$ has slope one on $e_x$ in the direction $\boldsymbol{l}$ with lattice length, and

$(5)$ $f$ is constant zero on $V \setminus e_x$.
\end{lemma}

\begin{proof}
By assumption, we can choose a point $x$ of $e$ and distinct vectors $\boldsymbol{s}_1, \ldots, \boldsymbol{s}_m \in \boldsymbol{R}^n \setminus \{(0, \ldots, 0)\}$ satisfying the conditions $(1)$, $(2)$ and $(3)$.

Let $(l_1, \ldots, l_n) := \boldsymbol{l}$.
For any $1 \le i, j \le n$, let $f_{ij}$ be the element $(x)_i^{\odot l_j} \odot (x)_j ^{\odot (-l_i)} \odot X_i^{\odot (-l_j)} \odot X_j^{\odot l_i}$ of $\overline{\boldsymbol{T}(X_1, \ldots, X_n)}$.
Each $f_{\ij}$ is constant zero on $\{ x + t\boldsymbol{l} \,|\, t \in \boldsymbol{R} \}$ and has the slope zero in the directions $\pm \boldsymbol{l}$ at any point of $\boldsymbol{R}^n$.
The slope of $g := \bigoplus_{i, j = 1}^n \left(f_{ij} \oplus f_{ij}^{\odot (-1)}\right)$ in any direction away from the line $\{ x + t\boldsymbol{l} \,|\, t \in \boldsymbol{R} \}$ at any point of $\boldsymbol{R}^n$ is positive.
Since $f_{ij} \oplus f_{ij}^{\odot (-1)} \ge 0$, we have $g \ge 0$ and $g$ takes the value zero on and only on $\{ x + t \boldsymbol{l} \,|\, t \in \boldsymbol{R} \}$.

Since $\boldsymbol{l}$ is primitive, there exists $\boldsymbol{l}^{\prime} := ( l_1^{\prime}, \ldots, l_n^{\prime}) \in \boldsymbol{Z}^n \setminus \{ (0, \ldots, 0) \}$ such that $\boldsymbol{l} \cdot \boldsymbol{l}^{\prime} = l_1 l_1^{\prime} + \cdots + l_n l_n^{\prime} = 1$.
Let $h$ be the element $(x)_1^{\odot (- l_1^{\prime})} \odot \cdots \odot (x)_n^{\odot (- l_n^{\prime})} \odot X_1^{\odot l_1^{\prime}} \odot \cdots \odot X_n^{\odot l_n^{\prime}} \oplus 0$ of $\overline{\boldsymbol{T}(X_1, \ldots, X_n)}$.
Then $h$ takes the value zero on $\{ x + y \,|\, y \in \boldsymbol{R}^n, y \cdot \boldsymbol{l}^{\prime} \le 0\}$ and $h$ has slope one in the direction $\boldsymbol{l}$ at any point of $\{ x + y \,|\, y \in \boldsymbol{R}^n, y \cdot \boldsymbol{l}^{\prime} \ge 0 \}$ with lattice length.
Let $U_x(1)$ be the $1$-neighborhood of $x$, i.e., $U_x(1) := \{ y \in \boldsymbol{R}^n \,|\, \operatorname{dist}(x, y) = 1 \}$, where $\operatorname{dist}(x, y) = 1$ means the distanse between $x$ and $y$ is one with the usual metric.
Let $L := \{ x + y \,|\, y \in \boldsymbol{R}^n, y \cdot \boldsymbol{l}^{\prime} \ge 0 \} \cap U_x(1)$.
Then, by replacing $\boldsymbol{s}_1, \ldots, \boldsymbol{s}_m$ with other vectors satisfying the conditions $(1), (2)$ and $(3)$ if we need, we may assume that $M := \overline{L \setminus (\{ x \} + \operatorname{Cone}\{ \boldsymbol{s}_1, \ldots, \boldsymbol{s}_m \})}$ is nonempty and compact.
Since both $g$ and $h$ are positive on $M$, there exists $k \in \boldsymbol{Z}_{> 0}$ such that  $h(y) \le g(y)^{\odot k}$ for any $y \in M$.
For any $t \in \boldsymbol{R}_{\ge0}$, since $h(x + t(y - x)) = th(y)$ and $f_{ij}(x + t(y - x)) = t f_{ij}(y)$ for any $1 \le i, j \le n$, we have $h(x + t(y - x)) \le g(x + t(y - x))^{\odot k}$.
Hence $f := h \odot g^{\odot (-k)}  \oplus 0 \in \overline{\boldsymbol{T}(X_1, \ldots, X_n)}$ is a desired tropical rational function.
\end{proof}

\begin{lemma}
	\label{lem10}
Let $e$ be any segment of $V$ of a finite length, $x, y$ its distinct endpoints, $z$ its midpoint and $l$ the length of $e$ with lattice length.
Then there exists $f \in \overline{\boldsymbol{T}(X_1, \ldots, X_n)}$ such that $f(z) = 0$, the slope of $f$ is one on the segments $\overline{xz}$ and $\overline{zy}$ with lattice length and $f$ is constant $-\frac{l}{2}$ on $V \setminus e$.
\end{lemma}

\begin{proof}
Assume that $n = 1$.
The tropical rational function 
\begin{align*}
\left\{ ((-z) \odot X_1) \oplus ((-z) \odot X_1)^{\odot (-1)} \right\}^{\odot (-1)} \oplus \left( - \frac{l}{2} \right)
\end{align*}
is desired $f$.

Assume that $n \ge 2$.
Let $\boldsymbol{l}$ be the primitive direction vector of $e$ from $x$ to $y$.
For $x$ (resp. $y$), by the same argument in the proof of Lemma \ref{lem9}, we can choose distinct vectors $\boldsymbol{s}_{x1}, \ldots, \boldsymbol{s}_{xm_x} \in \boldsymbol{R}^n \setminus \{ (0, \ldots, 0)\}$ (resp. $\boldsymbol{s}_{y1}, \ldots, \boldsymbol{s}_{ym_y} \in \boldsymbol{R}^n \setminus \{ (0, \ldots, 0)\}$) and $f_x \in \overline{\boldsymbol{T}(X_1, \ldots, X_n)}$ (resp. $f_y \in \overline{\boldsymbol{T}(X_1, \ldots, X_n)}$) satisfying the conditions $(1)$, $(2)$ for $\boldsymbol{l}$ (resp. $-\boldsymbol{l}$) and $(4)$ in Lemma \ref{lem9} and $f_x$ (resp. $f_y$) is constant zero on $V \setminus ( \{ x \} + \operatorname{Cone}\{ \boldsymbol{s}_{x1}, \ldots, \boldsymbol{s}_{xm_x} \})$ (resp. $V \setminus ( \{ y \} + \operatorname{Cone}\{ \boldsymbol{s}_{y1}, \ldots, \boldsymbol{s}_{ym_y} \})$.
By replacing $\boldsymbol{s}_{x1}, \ldots, \boldsymbol{s}_{xm_x}, \boldsymbol{s}_{y1}, \ldots, \boldsymbol{s}_{ym_y}$ (and hence $f_x$ and $f_y$) if we need, we can assume that $V \cap (\{ x \} + \operatorname{Cone}\{ \boldsymbol{s}_{x1}, \ldots, \boldsymbol{s}_{xm_x} \}) \cap (\{ y \} + \operatorname{Cone}\{ \boldsymbol{s}_{y1}, \ldots, \boldsymbol{s}_{ym_y} \}) = e$ holds.
Hence the tropical rational function $- \frac{l}{2} \odot \left( f_x^{\odot (-1)} \oplus f_y^{\odot (-1)} \right)^{\odot (-1)} \oplus \left(-\frac{l}{2} \right)$ is a desired one.
\end{proof}

\begin{lemma}
	\label{lem11}
For any $x \in V$, there exist $\varepsilon > 0$, $\operatorname{val}(x)$ points $y_i$ of $V$ such that the distance between $x$ and $y_i$ is $\varepsilon$ with lattice length and $i$ moves all integers from one to the valency $\operatorname{val}(x)$ of $x$ in $V$, and $f \in \overline{\boldsymbol{T}(X_1, \ldots, X_n)}$ which has slope one in the direction from $y_i$ to $x$ for any $i$ with lattice length and is constant $-\varepsilon$ on $V \setminus \bigcup_{i = 1}^{\operatorname{val}(x)}\overline{xy_i}$.
\end{lemma}

\begin{proof}
When $n = 1$, if $x$ has valency two in $V$, then the assertion follows from Lemma \ref{lem10}; if $x$ has valency one in $V$, then one of the tropical rational functions $-x \odot X_1 \oplus (-\varepsilon)$ and $x \odot X_1^{\odot (-1)} \oplus (-\varepsilon)$ is desired.

Assume that $n \ge 2$.
Let $\boldsymbol{l}_i$ be the primitive direction vector from $x$ to $y_i$.
For $x$ and $\boldsymbol{l}_i$, by the same argument in the proof of Lemma \ref{lem9}, we can choose distinct vectors $\boldsymbol{s}_{i1}, \ldots, \boldsymbol{s}_{im_i} \in \boldsymbol{R}^n \setminus \{ (0, \ldots, 0) \}$ and $f_i \in \overline{\boldsymbol{T}(X_1, \ldots, X_n)}$ satisfying the conditions $(1)$, $(2)$ and $(4)$ in Lemma \ref{lem9} and $f_i$ is constant zero on $V \setminus (\{ x \} + \operatorname{Cone}\{ \boldsymbol{s}_{i1}, \ldots, \boldsymbol{s}_{im_i} \})$.
By replacing $\boldsymbol{s}_{11}, \ldots, \boldsymbol{s}_{1m_1}, \boldsymbol{s}_{21}, \ldots, \boldsymbol{s}_{\operatorname{val}(x)m_{\operatorname{val}(x)}}$ if we need, we can assume that if $i \not= j$, then $y_i$ is not in $\{ x \} + \operatorname{Cone}\{ \boldsymbol{s}_{j1}, \ldots, \boldsymbol{s}_{jm_j}\}$.
Thus for a sufficiently small positive number $\varepsilon$, the tropical rational function $\bigodot_{i = 1}^{\operatorname{val}(x)} f_i^{\odot (-1)} \oplus (- \varepsilon)$ is desired.
\end{proof}

Let $V$ be a closed connected subset of $\boldsymbol{R}^n$ of geometric dimension one whose associated congruence $\boldsymbol{E}(V)$ is finitely generated as a congruence, or consist of only one point of $\boldsymbol{R}^n$.
Let $V^{\prime}$ be the natural compactification of $V$ as a tropical curve, i.e., let $V^{\prime}$ be obtained from $V$ by one-point compactifications of rays of $V$.

\begin{prop}
	\label{prop2}
In the above setting, $\operatorname{Rat}(V^{\prime})$ is isomorphic to $\overline{\boldsymbol{T}(X_1, \ldots, X_n)} / \boldsymbol{E}(V)$ as a $\boldsymbol{T}$-algebra if and only if there exist no two distinct rays of $V$ whose primitive direction vectors coincide.
\end{prop}

\begin{proof}
If $V$ consists of only one point of $\boldsymbol{R}^n$, then $\overline{\boldsymbol{T}(X_1, \ldots, X_n)} / \boldsymbol{E}(V)$ is isomorphic to $\boldsymbol{T} = \operatorname{Rat}(V^{\prime})$ and there exist no two distinct rays of $V$ whose primitive direction vectors coincide.
If $n = 1$, then the tropical rational function on $V$ represented by $X_1$ generates the semifield over $\boldsymbol{T}$ isomorphic to $\operatorname{Rat}(V^{\prime})$ as a $\boldsymbol{T}$-algebra by the proof of \cite[Proposition 3.13]{JuAe2}.
Hence we assume that $V$ does not consist of only one point and $n \ge 2$.

We shall show the if part.
Assume that there exist two distince rays $e_1, e_2$ of $V$ which have the same primitive direction vector $(s_1, \ldots, s_n) \in \boldsymbol{Z}^n \setminus \{ (0, \ldots, 0) \}$.
Let $f$ be an element of $\overline{\boldsymbol{T}(X_1, \ldots, X_n)} \setminus \{ -\infty \}$.
We can choose $f_1, f_2 \in \overline{\boldsymbol{T}[X_1, \ldots, X_n]} \setminus \{ -\infty \}$ such that $f = f_1 \odot f_2^{\odot (-1)}$.
Since both $f_1$ and $f_2$ are tropical polynomial functions other than $-\infty$, we can assume that 
\begin{align*}
&f_1(((x_1)_1 + ts_1, \ldots, (x_1)_1 + ts_n))\\
=~&a \odot ((x_1)_1 + ts_1)^{\odot i_1} \odot \cdots \odot ((x_1)_n + ts_n)^{\odot i_n}
\end{align*}
and
\begin{align*}
&f_2(((x_1)_1 + ts_1, \ldots, (x_1)_1 + ts_n))\\
=~&b \odot ((x_1)_1 + ts_1)^{\odot j_1} \odot \cdots \odot ((x_1)_n + ts_n)^{\odot j_n}
\end{align*}
with some $x_1 \in e_1$, $a, b \in \boldsymbol{R}$ and $i_1, \ldots, i_n, j_1, \ldots, j_n \in \boldsymbol{Z}_{\ge 0}$ for any $t \ge 0$.
By the same reason, there exists $x_2 \in e_2$ such that for any $t \ge 0$,
\begin{align*}
&f_1(((x_2)_1 + ts_1, \ldots, (x_2)_1 + ts_n))\\
=~&a^{\prime} \odot ((x_2)_1 + ts_1)^{\odot i^{\prime}_1} \odot \cdots \odot ((x_2)_n + ts_n)^{\odot i^{\prime}_n}
\end{align*}
and
\begin{align*}
&f_2(((x_2)_1 + ts_1, \ldots, (x_2)_n + ts_n))\\
=~&b^{\prime} \odot ((x_2)_1 + ts_1)^{\odot j^{\prime}_1} \odot \cdots \odot ((x_2)_n + ts_n)^{\odot j^{\prime}_n}
\end{align*}
with some $a^{\prime}, b^{\prime} \in \boldsymbol{R}$ and $i^{\prime}_1, \ldots, i^{\prime}_n, j^{\prime}_1, \ldots, j^{\prime}_n \in \boldsymbol{Z}_{\ge 0}$.
Then $i_1 s_1 + \cdots + i_n s_n$ (resp. $j_1 s_1 + \cdots + j_n s_n$) must coincide with $i^{\prime}_1 s_1 + \cdots + i^{\prime}_n s_n$ (resp. $j^{\prime}_1 s_1 + \cdots + j^{\prime}_n s_n$).
In fact, if $i_1 s_1 + \cdots + i_n s_n > i^{\prime}_1 s_1 + \cdots + i^{\prime}_n s_n$, then for a sufficiently large $t$, we have
\begin{align*}
&f_1(((x_2)_1 + ts_1, \ldots, (x_2)_n + ts_n))\\
\ge~&a \odot ((x_2)_1 + ts_1)^{\odot i_1} \odot \cdots \odot ((x_2)_1 + ts_n)^{\odot i_n}\\
=~&a + i_1(x_2)_1 + \cdots + i_n(x_2)_n + t (i_1 s_1 + \cdots + i_n s_n)\\
>~&a^{\prime} + i^{\prime}_1(x_2)_1 + \cdots + i^{\prime}_n(x_2)_n + t(i^{\prime}_1 s_1 + \cdots + i^{\prime}_n s_n)\\
=~&f_1(((x_2)_1 + ts_1, \ldots, (x_2)_n + ts_n)),
\end{align*}
which is a contradiction.
Thus $i_1 s_1 + \cdots + i_n s_n \le i^{\prime}_1 s_1 + \cdots + i^{\prime}_n s_n$.
By the same argument, we have $i_1 s_1 + \cdots + i_n s_n \ge i^{\prime}_1 s_1 + \cdots + i^{\prime}_n s_n$, and hence $i_1 s_1 + \cdots + i_n s_n = i^{\prime}_1 s_1 + \cdots + i^{\prime}_n s_n$.
Similarly, we have $j_1 s_1 + \cdots + j_n s_n = j^{\prime}_1 s_1 + \cdots + j^{\prime}_n s_n$.
Since
\begin{align*}
&&&f(((x_1)_1 + ts_1, \ldots, (x_1)_n + ts_n))\\
&&=~& f_1(((x_1)_1 + ts_1, \ldots, (x_1)_n + ts_n)) - f_2(((x_1)_1 + ts_1, \ldots, (x_1)_n + ts_n))\\
&&=~& (a - b) + (i_1 - j_1)(x_1)_1 + \cdots + (i_n - j_n)(x_1)_n\\
&&&+ t \{ (i_1 s_1 + \cdots + i_n s_n) - (j_1 s_1 + \cdots + j_n s_n) \}\\
&&=~& (a^{\prime} - b^{\prime}) + (i^{\prime}_1 - j^{\prime}_1)(x_2)_1 + \cdots + (i^{\prime}_n - j^{\prime}_n)(x_2)_n\\
&&&+ t \{ (i_1 s_1 + \cdots + i_n s_n) - (j_1 s_1 + \cdots + j_n s_n) \}\\
&&~& + \{(a - b) - (a^{\prime} - b^{\prime}) \} + \{ (i_1 - j_1)(x_1)_1 - (i^{\prime}_1 - j^{\prime}_1)(x_2)_1 \} + \cdots \\
&&~&+ \{ (i_n - j_n)(x_1)_n - (i^{\prime}_n - j^{\prime}_n)(x_2)_n \}\\
&&=~& f_1(((x_2)_1 + ts_1, \ldots, (x_2)_n + ts_n)) - f_2(((x_2)_1 + ts_1, \ldots, (x_2)_n + ts_n))\\
&&~& + \{(a - b) - (a^{\prime} - b^{\prime}) \} + \{ (i_1 - j_1)(x_1)_1 - (i^{\prime}_1 - j^{\prime}_1)(x_2)_1 \} + \cdots \\
&&~&+ \{ (i_n - j_n)(x_1)_n - (i^{\prime}_n - j^{\prime}_n)(x_2)_n \}\\
&&=~& f(((x_2)_1 + ts_1, \ldots, (x_2)_n + ts_n))\\
&&& + \{(a - b) - (a^{\prime} - b^{\prime}) \} + \{ (i_1 - j_1)(x_1)_1 - (i^{\prime}_1 - j^{\prime}_1)(x_2)_1 \} + \cdots \\
&&&+ \{ (i_n - j_n)(x_1)_n - (i^{\prime}_n - j^{\prime}_n)(x_2)_n \},
\end{align*}
$f(((x_1)_1 + ts_1, \ldots, (x_1)_n + ts_n)) \to \pm \infty$ as $t \to \infty$ if and only if $f(((x_2)_1 + ts_1, \ldots, (x_2)_n + ts_n)) \to \pm \infty$ as $t \to \infty$.
If $\psi$ is a $\boldsymbol{T}$-algebra isomorphism from $\overline{\boldsymbol{T}(X_1, \ldots, X_n)} / \boldsymbol{E}(V)$ to $\operatorname{Rat}(V^{\prime})$, then $g_1 := \psi(\pi_{\boldsymbol{E}(V)}(X_1)), \ldots, g_n := \psi(\pi_{\boldsymbol{E}(V)}(X_n)) \in \operatorname{Rat}(V^{\prime}) \setminus \{ -\infty \}$ generate $\operatorname{Rat}(V^{\prime})$ as a semifield over $\boldsymbol{T}$, where $\pi_{\boldsymbol{E}(V)} : \overline{\boldsymbol{T}(X_1, \ldots, X_n)} \twoheadrightarrow \overline{\boldsymbol{T}(X_1, \ldots, X_n)} / \boldsymbol{E}(V)$ is the natural surjective $\boldsymbol{T}$-algebra homomorphism.
By \cite[Proposition 3.12(1) and (5)]{JuAe4}, we have the isometry $\theta : V^{\prime} \setminus V^{\prime}_{\infty} \to V; x \mapsto (g_1(x), \ldots, g_n(x))$ and $\psi(\pi_{\boldsymbol{E}(V)}(g))(x) = g(\theta(x))$ for any $g \in \overline{\boldsymbol{T}(X_1, \ldots, X_n)}$ and $x \in V^{\prime} \setminus V^{\prime}_{\infty}$.
On the other hand, $V^{\prime}$ has a rational function that has the slope one on the edge $e_1^{\prime}$ corresponding to $e_1$ and is constant on the whole of $V^{\prime} \setminus e_1^{\prime}$, which is a contradiction.
This means that $\operatorname{Rat}(V^{\prime})$ is not isomorphic to $\overline{\boldsymbol{T}(X_1, \ldots, X_n)} / \boldsymbol{E}(V)$ as a $\boldsymbol{T}$-algebra.

We shall show the only if part.
By Lemma \ref{lem9}, we can choose rays $e_1, \ldots, e_m$ in $V$ and $f_{e_1}, \ldots, f_{e_m} \in \overline{\boldsymbol{T}(X_1, \ldots, X_n)}$ such that each $f_{e_i}$ has slope one on $e_i$ with lattice length and is constant zero on $V \setminus e_i$ and $\overline{V \setminus \bigcup_{i = 1}^m e_i}$ is compact.
Since $\boldsymbol{E}(\overline{V \setminus \bigcup_{i = 1}^m e_i})$ is finitely generated as a congruence by Theorem \ref{thm:main} and $\overline{V \setminus \bigcup_{i = 1}^m e_i}$ is a closed connected subset of $\boldsymbol{R}^n$ of geometric dimension zero or one, by Lemmas \ref{lem10} and \ref{lem11} and the proof of \cite[Proposition 3.6]{JuAe2}, the semifield over $\boldsymbol{T}$ generated by all $f \in \overline{\boldsymbol{T}(X_1, \ldots, X_n)}$ in Lemmas \ref{lem10} and \ref{lem11} is isomorphic to $\operatorname{Rat}(\overline{V \setminus \bigcup_{i = 1}^m e_i})$ as a $\boldsymbol{T}$-algebra.
Note that $\operatorname{Rat}(\overline{V \setminus \bigcup_{i = 1}^m e_i})$ is the rational function semifield of $\overline{V \setminus \bigcup_{i = 1}^m e_i}$ when we regard $\overline{V \setminus \bigcup_{i = 1}^m e_i}$ as a tropical curve with lattice length.
By the proof of \cite[Lemma 1.4]{JuAe2}, we know that all such $f$s and $f_{e_1}, \ldots, f_{e_m}$ generate a semifield over $\boldsymbol{T}$ isomorphic to $\operatorname{Rat}(V^{\prime})$ as a $\boldsymbol{T}$-algebra.
Since there exists a natural $\boldsymbol{T}$-algebra homomorphism from $\overline{\boldsymbol{T}(X_1, \ldots, X_n)}$ to $\operatorname{Rat}(V^{\prime})$ by regarding each element of $\overline{\boldsymbol{T}(X_1, \ldots, X_n)}$ as an element of $\operatorname{Rat}(V^{\prime})$, we have the conclusion.
\end{proof}

In the above setting, by \cite[Proposition 3.12(5)]{JuAe4}, if there exists a tropical curve $\Gamma$ such that $\operatorname{Rat}(\Gamma)$ is isomorphic to $\overline{\boldsymbol{T}(X_1, \ldots, X_n)} / \boldsymbol{E}(V)$ as a $\boldsymbol{T}$-algebra, then $\Gamma$ is isomorphic to $V^{\prime}$, and hence $\operatorname{Rat}(V^{\prime})$ is also isomorphic to $\overline{\boldsymbol{T}(X_1, \ldots, X_n)} / \boldsymbol{E}(V)$.
Thus by \cite[Theorem 1.1]{JuAe2}, Corollary \ref{cor8}, \cite[Proposition 3.12]{JuAe4} and Proposition \ref{prop2}, we have the following corollary, which characterizes the rational function semifields of tropical curves:

\begin{cor}
	\label{cor9}
Let $S$ be a semifield which is finitely generated as a semifield over $\boldsymbol{T}$.
Then $S$ is isomorphic to $\operatorname{Rat}(\Gamma)$ as a $\boldsymbol{T}$-algebra with some tropical curve $\Gamma$ if and only if some (and hence, any) surjective $\boldsymbol{T}$-algebra homomorphism $\psi$ from a tropical rational function semifield to $S$ satisfies the following five conditions:

$(1)$ $\operatorname{Ker}(\psi)$ is finitely generated as a congruence,

$(2)$ $\operatorname{Ker}(\psi) = \boldsymbol{E}(\boldsymbol{V}(\operatorname{Ker}(\psi)))$ holds,

$(3)$ $\boldsymbol{V}(\operatorname{Ker}(\psi))$ is connected,

$(4)$ $\boldsymbol{V}(\operatorname{Ker}(\psi))$ is of geometric dimension zore or one, and

$(5)$ $\boldsymbol{V}(\operatorname{Ker}(\psi))$ has no two distinct rays whose primitive direction vectors coincide.
\end{cor}

By Corollary \ref{cor9}, \cite[Corollary 1.2]{JuAe3} and \cite[Corollary 3.18]{JuAe4}, we have the following corollary.
This is a tropical analogue of the fact that the category of nonsingular projective curves and dominant morphisms is equivalent to the category of function fields of dimension one over $k$ and $k$-homomorphisms, where $k$ is an algebraically closed field (see \cite[2, Corollary 6.12. in Chapter I]{Hartshorne}).

\begin{cor}
	\label{cor10}
The following categories $\mathscr{C}, \mathscr{D}$ are equivalent via the following (contravariant) functors $F, G$.

$(1)$ The class $\operatorname{Ob}(\mathscr{C})$ of objects of $\mathscr{C}$ is the tropical curves.

For $\Gamma_1, \Gamma_2 \in \operatorname{Ob}(\mathscr{C})$, the set $\operatorname{Hom}_{\mathscr{C}}(\Gamma_1, \Gamma_2)$ of morphisms from $\Gamma_1$ to $\Gamma_2$ consists of the surjective morphisms $\Gamma_1 \twoheadrightarrow \Gamma_2$.

$(2)$ The class $\operatorname{Ob}(\mathscr{D})$ of objects of $\mathscr{D}$ is the finitely generated semifields over $\boldsymbol{T}$ satisfying the conditions $(1), \ldots, (5)$ in Corollary \ref{cor9}.

For $S_1, S_2 \in \operatorname{Ob}(\mathscr{D})$, the set $\operatorname{Hom}_{\mathscr{D}}(S_1, S_2)$ of morphisms from $S_1$ to $S_2$ consists of the injective $\boldsymbol{T}$-algebra homomorphisms $S_1 \hookrightarrow S_2$.

The functor $F : \mathscr{C} \to \mathscr{D}$ maps $\Gamma \in \operatorname{Ob}(\mathscr{C})$ to $\operatorname{Rat}(\Gamma)$ and for $\Gamma_1, \Gamma_2 \in \operatorname{Ob}(\mathscr{C})$, maps $\varphi \in \operatorname{Hom}_{\mathscr{C}}(\Gamma_1, \Gamma_2)$ to $\varphi^{\ast} \in \operatorname{Hom}_{\mathscr{D}}(\operatorname{Rat}(\Gamma_2), \operatorname{Rat}(\Gamma_1))$.

The functor $G : \mathscr{D} \to \mathscr{C}$ maps $S \in \operatorname{Ob}(\mathscr{D})$ to $\boldsymbol{V}(\operatorname{Ker}(\psi_S))^{\prime} \in \operatorname{Ob}(\mathcal{C})$ and for $S_1, S_2 \in \operatorname{Ob}(\mathscr{D})$, maps $\widetilde{\psi} \in \operatorname{Hom}_{\mathscr{D}}(S_1, S_2)$ to $\varphi_{\widetilde{\psi}} \in \operatorname{Hom}_{\mathscr{C}}(\boldsymbol{V}(\operatorname{Ker}(\psi_{S_2}))^{\prime}, \boldsymbol{V}(\operatorname{Ker}(\psi_{S_1}))^{\prime})$, where $\psi_S$ is a fixed surjective $\boldsymbol{T}$-algebra homomorphism from a tropical rational function semifield to $S$ and $\boldsymbol{V}(\operatorname{Ker}(\psi_S))^{\prime}$ is the natural compactification of $\boldsymbol{V}(\operatorname{Ker}(\psi_S))$ as a tropical curve for any $S \in \operatorname{Ob}(\mathscr{D})$ and $\varphi_{\widetilde{\psi}}$ is the unique surjective morphism $\boldsymbol{V}(\operatorname{Ker}(\psi_{S_2}))^{\prime} \twoheadrightarrow \boldsymbol{V}(\operatorname{Ker}(\psi_{S_1}))^{\prime}$ such that $\widetilde{\psi} = \left(\varphi_{\widetilde{\psi}}\right)^{\ast}$ given in \cite[Theorem 1.1]{JuAe3} and \cite[Corollary 3.18]{JuAe4}.
\end{cor}

Let $\Gamma$ be a tropical curve and $D$ an effective divisor on $\Gamma$.
Let $F := \{ f_1, \ldots, f_n \}$ be a finite generating set of $R(D)$ as a $\boldsymbol{T}$-module such that $f_i \not= -\infty$ for all $i$ and $\phi_F: \Gamma \setminus \Gamma_{\infty} \to \boldsymbol{R}^n; x \mapsto (f_1(x), \cdots, f_n(x))$ the rational map induced by $F$.
Note that $F$ contains all extremals of $R(D)$ up to tropical scalar multiplication (other than $-\infty$).

\begin{cor}
	\label{cor11}
In the above setting, if $\phi_F$ is injective, then $\operatorname{Rat}(\Gamma)$ is isomorphic to $\overline{\boldsymbol{T}(X_1, \ldots, X_n)} / \boldsymbol{E}(\operatorname{Im}(\phi_F))$ as a $\boldsymbol{T}$-algebra.
In particular, if $\operatorname{deg}(D) \ge 2g(\Gamma) + 1$, then $\phi_F$ is injective.
\end{cor}

\begin{proof}
If $\Gamma$ is a singleton, then the assertion is clear.
Assume that $\Gamma$ is not a singleton.
By the definition of $\phi_F$ and Theorem \ref{thm:main}, $\operatorname{Im}(\phi_F)$ is a closed subset of $\boldsymbol{R}^n$ of geometric dimension one whose associated congruence $\boldsymbol{E}(\operatorname{Im}(\phi_F))$ is finitely generated as a congruence.
By \cite[Lemma 3.3.10]{JuAe1}, if $\phi_F$ is injective, then $\phi_F$ is a local isometry on the image.
If $F$ consists of only extremals, i.e., $F$ is minimal, and the statement holds, then it is also true when $F$ is an arbitrary finite generating set of $R(D)$ since $\operatorname{Im}(\phi_F)$ cannot have two distinct rays whose primitive direction vectors coincide.
Hence, without loss of generality, we may assume that $F$ has only extremals.
Assume that there exist two distinct rays $e_1$ and $e_2$ of $\operatorname{Im}(\phi_F)$ which have the same primitive direction vector $(s_1, \ldots, s_n) \in \boldsymbol{Z}^n \setminus \{ (0, \ldots, 0) \}$.
Let $e_i^{\prime}$ be the (closed) edge of $\Gamma$ corresponding to $e_i$ and $x_i^{\prime}$ the endpoint of $e_i^{\prime}$ at infinity.
By assumption, for any $j$, the slopes of $f_j$ on $e_1^{\prime}$ and $e_2^{\prime}$ in the direction to $x_1^{\prime}$ and $x_2^{\prime}$ are $s_j$.
By \cite[Lemma 3.3.10]{JuAe1} again, there exists at least one $j$ such that $s_j$ is one or minus one.
In particular, if $\Gamma$ is a tree, there exists $j$ such that $s_j = -1$.
Then we have $(D + \operatorname{div}(f_j))(x_1^{\prime}) = D(x_1^{\prime}) + \operatorname{div}(f_j)(x_1^{\prime}) > 0$ and for a rational function $f$ on $\Gamma$ which has slope one from $x_1^{\prime}$ to $x_2^{\prime}$ and is constant elsewhere, the rational function $f \odot f_j$ is an extremal of $R(D)$ by \cite[Lemma 3.11]{Haase=Musiker=Yu} and this has different slopes on $e_1^{\prime}$ and $e_2^{\prime}$ in the directions to $x_1^{\prime}$ and $x_2^{\prime}$, respectively.
This is a contradiction. 

Assume that $\Gamma$ is not a tree.
Let $x_0^{\prime}$ be the nearest point to $x_1^{\prime}$ whose valency is at least three.
We first consider the case $s_j = -1$.
Then we have $D(x_1^{\prime}) + \operatorname{div}(f_j)(x_1^{\prime}) > 0$ again and a rational function $g$ on $\Gamma$ which has slope one from $x_1^{\prime}$ to $x_0^{\prime}$, the rational function $g \odot f_j$ is an extremal of $R(D)$ by \cite[Lemma 3.11]{Haase=Musiker=Yu} and this has different slopes on $e_1^{\prime}$ and $e_2^{\prime}$ in the directions to $x_1^{\prime}$ and $x_2^{\prime}$, respectively, which is a contradiction.
We next consider the case $s_j = 1$.
Let $T$ be the maximum subgraph of $\Gamma$ which is a tree and contains $x_1^{\prime}$.
Then there exists a rational function $h \in R(D + \operatorname{div}(f_j))$ such that $D + \operatorname{div}(f_j) + \operatorname{div}(h)$ is zero on $T \setminus \{ x_1^{\prime} \}$ and $h$ is constant on $\Gamma \setminus T$.
Since $D(x_1^{\prime}) + \operatorname{div}(f_j)(x_1^{\prime}) > 0$, the rational function $g^{\odot (-1)} \odot h \odot f_j$ is an extremal of $R(D)$ and has different slopes on $e_1^{\prime}$ and $e_2^{\prime}$ in the dirextions to $x_1^{\prime}$ and $x_2^{\prime}$, respectively.
It is a contradiction.

Therefore, if $\phi_F$ is injective, then there exist no distinct rays on $\operatorname{Im}(\phi_F)$ which have the same primitive direction vector, and hence we have the conclusion by Proposition \ref{prop2}.
The last assertion follows from \cite[Theorem 50 and Lemma 41]{Haase=Musiker=Yu}.
\end{proof}


\begin{thebibliography}{9}

\bibitem{Bergman} George Bergman, \textit{The logarithmic limit-set of an algebraic variety}, Transactions of the American Mathematical Society \textbf{157}:459--469, 1971.

\bibitem{Bertram=Easton} Aaron Bertram and Robert Easton, \textit{The tropical Nullstellensatz for congruences}, Advances in Mathematics \textbf{308}:36--82, 2017.

\bibitem{Chan} Melody Chan, \textit{Tropical hyperelliptic curves}, Journal of Algebraic Combinatorics \textbf{37}:331-359, 2013.

\bibitem{CDPR} Filip Cools, Jan Draisma, Sam Payne and Elina Robeva, \textit{A tropical proof of the Brill--Noether theorem}, Advances in Mathematics \textbf{230}(2):759--776, 2012.

\bibitem{GKZ} Israel M.~Geldand, Mikhail M.~Kapranov and Andrei V.~Zelevinsky, \textit{Discriminants, Resultants, and Multidimensional Determinants}, Birkh\"{a}user Boston, 1994.

\bibitem{Giansiracusa=Giansiracusa1} Jeffrey Giansiracusa and Noah Giansiracusa, \textit{Equations of tropical varieties}, Duke Mathematical Journal \textbf{165}(18):3379--3433, 2016.

\bibitem{Giansiracusa=Giansiracusa2} Jeffrey Giansiracusa and Noah Giansiracusa, \textit{The universal tropicalization and the Berkovich analytification}, Kybemetika \textbf{58}(5):790--815, 2022.

\bibitem{Golan} Jonathan S.~Golan, \textit{Semirings and their applications}, Updated and expanded version of The theory of semirings, with applications to mathematics and theoretical computer science. Kluwer Academic Publishers, Dordrecht, 1999.

\bibitem{Grigoriev} Dima Grigoriev, \textit{A tropical version of Hilbert function}, arXiv:2404.06440.

\bibitem{Haase=Musiker=Yu} Christian Haase, Gregg Musiker and Josephine Yu, \textit{Linear Systems on Tropical Curves}, Mathematische Zeitschrift {\bf 270}, no.~3-4, 1111--1140, 2012.

\bibitem{Hartshorne} Robin Hartshorne, \textit{Algebraic geometry}, Graduate Texts in Mathematics, No. 52. Springer-Verlag, New York-Heidelberg, 1977.

\bibitem{JuAe1} Song JuAe, \textit{Galois quotients of tropical curves and invariant linear systems}, Hokkaido Mathematical Journal \textbf{51}(3):445-486, 2022.

\bibitem{JuAe2} Song JuAe, \textit{Rational function semifields of tropical curves are finitely generated over the tropical semifield}, International Journal of Algebra and Computation \textbf{32}(08):1575--1594, 2022.

\bibitem{JuAe3} Song JuAe, \textit{$\boldsymbol{T}$-algebra homomorphisms between rational function semifields of tropical curves}, arXiv:2304.03508.

\bibitem{JuAe4} Song JuAe, \textit{Congruences on tropical rational function semifields of tropical curves}, arXiv:2305.04204v2.

\bibitem{JuAe5} Song JuAe, \textit{Semiring isomorphisms between rational function semifields of tropical curves}, Journal of Pure and Applied Algebra, \textbf{228}(9), 107683, 2024.

\bibitem{Joo=Mincheva1} D\'aniel Jo\'o and Kalina Mincheva, \textit{Prime congruences of additively idempotent semirings and a Nullstellensatz for tropical polynomials}, Selecta Mathematica \textbf{24}(3):2207--2233, 2018.

\bibitem{Joo=Mincheva2} D\'aniel Jo\'o and Kalina Mincheva, \textit{On the dimension of polynomial semirings} Journal of Algebra 507:103-119, 2018.

\bibitem{Kobayashi=Odagiri} Masanori Kobayashi and Shinsuke Odagiri, \textit{Tropical geometry of PERT}, Journal of Math-for-Industry \textbf{5}(B):145--149, 2013.

\bibitem{Maclagan=Sturmfels} Diane Maclagan and Bernd Sturmfels, \textit{Introduction to tropical geometry}, Graduate Studies in Mathematics, Vol.~161. American Mathematical Soc., Providence, RI, 2015.

\bibitem{Mikhalkin} Grigory Mikhalkin, \textit{Enumerative Tropical Algebraic Geometry in $\boldsymbol{R}^2$}, Journal of the American Mathematical Society \textbf{18}(2):313--377, 2005. 

\bibitem{Simon} Imre Simon, \textit{Recobnizable Sets with Multiplicities in the Tropical Semiring},  International Symposium on Mathematical Foundations of Computer Science. Berlin, Heidelberg: Springer Berlin Heidelberg, 1988.

\bibitem{Sturmfels} Bernd Sturmfels, \textit{Solving Systems of polynomial equations}, Vol.~97. American Mathematical Society, 2002.

\bibitem{Viro} Oleg Viro, \textit{Patchworking real algebraic varieties}, arXiv:0611382.
\end{thebibliography}
\end{document}